\documentclass[times,sort&compress, 3p]{elsarticle}
\journal{arXiv}
\usepackage[labelfont=bf]{caption}
\usepackage{bm}
\usepackage{ushort}

\usepackage{amsmath,amsfonts,amssymb,amsthm,booktabs,color,epsfig,graphicx,hyperref,url}
\usepackage{dsfont}
\usepackage{diagbox}
\usepackage{xcolor}

\theoremstyle{plain}
\newtheorem{theorem}{Theorem}

\newtheorem{proposition}{Proposition}
\newtheorem{lemma}{Lemma}
\newtheorem{corollary}{Corollary}

\theoremstyle{definition}

\newtheorem{remark}{Remark}

\newcommand{\pRz}{\mathbb R_+}
\newcommand{\IRp}{\pRz}
\newcommand{\IR}{\Rz}
\newcommand{\IL}{\mathbb L}
\newcommand{\Uz}{\mathbb U}
\newcommand{\sM}{\mathcal M}

\newcommand{\Rz}{\mathbb R}

\newcommand{\E}{\mathbb E}
\newcommand{\Var}{\mathbb V\mathrm{ar}}

\newcommand{\nset}[1]{{\left\llbracket #1\right\rrbracket }}

\DeclareMathOperator*{\argmin}{arg\,min}
\newcommand{\So}{f} 

\DeclareMathSymbol{\minus}{\mathbin}{AMSa}{"39}
\newcommand{\highlight}{\textit}

\begin{document}

\begin{frontmatter}

\title{Multiplicative deconvolution estimator based on a ridge approach}

\author[1]{Sergio Brenner Miguel \corref{mycorrespondingauthor}}

\address[1]{Institut f\"ur Angewandte
	Mathematik, M$\Lambda$THEM$\Lambda$TIKON, Im Neuenheimer Feld 205,
	D-69120 Heidelberg, Germany}

\cortext[mycorrespondingauthor]{Corresponding author. Email address: \url{brennermiguel@math.uni-heidelberg.de}}

\begin{abstract}
We study the non-parametric estimation of an unknown density $f$ with support on $\mathbb R_+$ based on an i.i.d. sample with multiplicative measurement errors. 
The proposed fully-data driven procedure consists of  the estimation of the Mellin
transform of the density $f$ and a regularisation of the inverse of
the Mellin transform by a ridge approach. The upcoming bias-variance trade-off is dealt with by a
data-driven  choice of the ridge parameter. In order to discuss the bias term, we consider the
\textit{Mellin-Sobolev spaces} which characterise the
regularity of the unknown density $f$ through the decay of its Mellin
transform. Additionally,  we show minimax-optimality over \textit{Mellin-Sobolev spaces} of the
ridge density estimator. 
\end{abstract}

\begin{keyword} 
Density estimation \sep
multiplicative measurement errors \sep
Mellin transform\sep
ridge estimator \sep
 minimax theory\sep
 inverse problem \sep
 adaptation
\MSC[2020] Primary 62G05 \sep secondary 62G07, 62C20
\end{keyword}

\end{frontmatter}

\section{Introduction\label{i}}

In this paper we are interested in estimating the unknown density
$f:\IRp \rightarrow \IRp$ of
a positive random variable $X$ given  independent and identically
distributed (i.i.d.) copies of $Y=XU$, where $X$ and $U$ are
independent of each other and $U$ has a known density $g:\IRp
\rightarrow \IRp$. In this setting the density  $f_{Y}:\IRp
\rightarrow \pRz$ of $ Y$ is given by
\begin{equation*}
f_{ Y}(y)=(f * g)(y):= \int_{\pRz} f(x)g( y/ x)x^{-1}d x\quad\forall  y\in\pRz.
\end{equation*}
 Here $*$ denotes multiplicative convolution. The estimation of
$f$ using an i.i.d. sample $ Y_1, \dots,  Y_n$ from $f_{ Y}$ is thus an
inverse problem called
multiplicative deconvolution.  \\
This particular model was studied by \cite{Brenner-MiguelComteJohannes2020}. Inspired by the work \cite{BelomestnyGoldenshluger2020}, the authors of \cite{Brenner-MiguelComteJohannes2020} introduced an estimator based on the estimation of the Mellin transform of the unknown density $f$ and a spectral cut-off regularisation of the inverse of the Mellin transform. In \cite{BelomestnyGoldenshluger2020} a pointwise kernel density estimator was proposed and investigated, while the authors of \cite{Brenner-MiguelComteJohannes2020} studied the global risk of the density estimation. For the model of multiplicative measurement, the multivariate case of global density estimaton, respectively the univariate case of global survival function estimation, was considered by \cite{Brenner-Miguel2021}  , respectively \cite{Brenner-MiguelPhandoidaen2021},  based on a spectral cut-off approach. \\
In this work, we will borrow the ridge approach from the additive deconvolution literature, for instance used by \cite{HallMeister2007} and \cite{Meister2009}, to build a new density estimator and compare it with the spectral cut-off estimator proposed by \cite{Brenner-MiguelComteJohannes2020}. The contribution of this work to the existing literature is the inclusion of the ridge approach and the comparison to the spectral cut-off approach. We discuss in which situations the corresponding estimators are comparable, respectively when the ridge approach is favourable. Furthermore, the ridge approach can be used  for furture works considering oscillatory error densities or unknown error densities, compare  \cite{HallMeister2007} and \cite{Meister2009}.
\subsection{Related works}
 The model of multiplicative measurement errors was motivated in the work of \cite{BelomestnyGoldenshluger2020} as a generalisation of several models, for instance the multiplicative censoring model or the stochastic volatility model. \\
\cite{Vardi1989} and \cite{VardiZhang1992} introduce and analyse
intensively \textit{multiplicative censoring}, which corresponds to
the particular multiplicative deconvolution problem with
multiplicative error $U$ uniformly distributed on $[0,1]$.  This
model is often applied in survival analysis as explained and
motivated in \cite{EsKlaassenOudshoorn2000}. The estimation of
the cumulative distribution function of $X$ is discussed in
\cite{VardiZhang1992} and
\cite{AsgharianWolfson2005}. Series expansion methods are
studied in \cite{AndersenHansen2001} treating the model as an
inverse problem. The density estimation in a multiplicative
censoring model is considered in \cite{BrunelComteGenon-Catalot2016}
using a kernel estimator and a convolution power kernel
estimator. Assuming an uniform error distribution on an interval
$[1-\alpha, 1+\alpha]$ for $\alpha\in (0,1)$, \cite{ComteDion2016}
analyse a projection density estimator with respect to the Laguerre
basis.  \cite{BelomestnyComteGenon-Catalot2016} investigate a
beta-distributed error $U$.\\
In the work of \cite{BelomestnyGoldenshluger2020}, the authors used the Mellin transform to construct a kernel estimator for the pointwise density estimation.  Moreover, they point out that the following widely used
naive approach is a special case of their estimation
strategy. Transforming the data by applying the logarithm to the model
$Y=XU$ writes as $\log(Y)=\log(X)+\log(U)$. In other words,
multiplicative convolution becomes convolution for the
$\log$-transformed data. As a consequence, the density of $\log(X)$
is eventually estimated employing usual strategies for
non-parametric deconvolution problems (see for example
\cite{Meister2009}) and then transformed back to an estimator of $f$.
However, it is difficult to interpret regularity conditions on the
density of $\log(X)$. Furthermore, the analysis of the global risk of
an estimator using this naive approach is challenging as
\cite{ComteDion2016}
pointed out. 
\subsection{Organisation}
The paper is organised as follows. In Section \ref{i} we recapitulate the definition of the Mellin transform and collect its frequently used properties. To be able to compare our estimator with the spectral cut-off estimator proposed by  \cite{Brenner-MiguelComteJohannes2020}, we will revisit its construction and state the necessary assumption for the estimator to be well-defined and present the ridge density estimator. In Section \ref{mt} we will show that the ridge density estimator is minimax-optimal over the Mellin-Sobolev spaces, by stating an upper bound and using the lower bound result given in \cite{Brenner-Miguel2021}. A data-driven procedure based on a Goldenshluger-Lepski method is described and analysed in Section \ref{dd}. Finally, results of a simulation
study are reported in section \ref{si} which visualize the reasonable finite
sample performance of our estimators. The proofs of Section \ref{mt} and
Section \ref{dd} are postponed to the Appendix.
\subsection{The spectral cut-off and ridge estimator}
We define for any weight function
$\omega:\IR \rightarrow \IR_+$ the corresponding
weighted norm  by $\|h\|_{\omega}^2 := \int_0^{\infty}
|h(x)|^2\omega(x)dx $ for a measurable, complex-valued function $h$. Denote by
$\IL^2(\IR_+,\omega)$ the set of all measurable, complex-valued functions with
finite $\|\, .\,\|_{\omega}$-norm and by $\langle h_1, h_2
\rangle_{\omega} := \int_0^{\infty}  h_1(x)
\overline{h_2}(x)\omega(x)dx$ for $h_1, h_2\in \IL^2(\IR_+,\omega)$
the corresponding weighted scalar product. Similarly, define $\IL^2(\IR):=\{ h:\IR \rightarrow \mathbb C\, \text{ measurable }: \|h\|_{\IR}^2:= \int_{-\infty}^{\infty} h(t)\overline{h(t)} dt <\infty \}$ and $\IL^1(\Omega,\omega):=\{h: \Omega \rightarrow \mathbb C: \|h\|_{\IL^1(\Omega,\omega)}:= \int_{\Omega} |h(x)|\omega(x)dx < \infty \}.$
In the introduction we already mentioned that the density $f_Y$ of $Y_1$ can be written as the multiplicative convolution of the densities $f$ and $g$. We will now define this convolution in a more general setting. Let $c\in \IR$. For two functions $h_1,h_2\in \IL^1(\IR_+, x^{c-1})$, where we use the notation $x^{c-1}$ for the weight function $x\mapsto x^{c-1}$, we define the multiplicative convolution $h_1*h_2$ of $h_1$ and $h_2$ by
\begin{align}\label{eq:mult:con}
(h_1*h_2)(y):=\int_0^{\infty} h_1(y/x) h_2(x) x^{-1} dx, \quad y\in \IR_+.
\end{align}
In fact, one can show that the function $h_1*h_2$ is well-defined, $h_1*h_2=h_2*h_1$ and $h_1*h_2 \in \IL^1(\IR,x^{c-1})$, compare \cite{Brenner-Miguel2021}. It is worth pointing out, that the definition of the multiplicative convolution in equation \ref{eq:mult:con} is independent of the model parameter $c\in \IR$. We also know for densities $h_1,h_2$ that $h_1,h_2\in \IL^1(\IR_+,x^0)$. If additionally $h_1\in \IL^2(\IR_+, x^{2c-1})$  then $h_1*h_2 \in \IL^2(\IR_+,x^{2c-1})$. 

\paragraph{Mellin transform}
We will now define the Mellin transform for $\IL^1(\IRp, x^{c-1})$ functions and present the convolution theorem. Further properties of the Mellin transform, which will be used in the upcoming theory, are collected in \ref{a:prel}. Proof sketches of these properties can be found in \cite{Brenner-MiguelComteJohannes2020}, respectively \cite{Brenner-Miguel2021}. Let $h_1\in \IL^1(\IR,x^{c-1})$. Then, we define the Mellin transform of $h_1$ at the development point $c\in \IR$ as the function $\sM_c[h]:\IR\rightarrow \mathbb C$ with
\begin{align}
\sM_c[h_1](t):= \int_0^{\infty} x^{c-1+it} h_1(x)dx, \quad t\in \IR.
\end{align}
The key property of the Mellin transform, which makes it so appealing for the use of multiplicative deconvolution, is the so-called convolution theorem, that is, for $h_1, h_2\in \IL^1(\IR_+,x^{c-1})$,
\begin{align}\label{eq:conv}
\sM_c[h_1*h_2](t)=\sM_c[h_1](t) \sM_c[h_2](t), \quad t\in \IR.
\end{align}
Let us now revisit the definition of the spectral cut-off estimator.
\paragraph{Spectral-cut off estimator} The family of spectral cut-off estimator $(\widetilde f_k)_{k\in \IRp}$ proposed by \cite{Brenner-MiguelComteJohannes2020}, respectively \cite{Brenner-Miguel2021}, is based on the estimation of the Mellin transform of $f_Y$ and a spectral cut-off regularisation of the inverse Mellin transform.  Given the sample $(Y_j)_{j\in \nset{n}}$, where $\llbracket a\rrbracket:=[1, a] \cap \mathbb N$ for any $a\in \mathbb N$, an unbiased estimator of $\sM_c[f_Y](t), t\in \IR,$ is given by the empirical Mellin transform $$\widehat{\sM}_c(t):=n^{-1} \sum_{j\in \nset{n}} Y_j^{c-1+it}, \quad t\in \mathbb R$$ if $\E_{f_Y}(Y_1^{c-1})< \infty$ for $c\in \IR$. Exploiting the convolution theorem, eq. \eqref{eq:conv}, under the assumption that $\sM_c[g](t)\neq 0$ we can define the unbiased estimator $\widehat{\sM}_c(t)/\sM_c[g](t)$ of $\sM_c[f](t)$ for $t\in \IR.$ To construct an estimator of the unknown density $f$, the authors of \cite{Brenner-MiguelComteJohannes2020} used a spectral-cut off approach. That is, for $k\in \IRp$ we assume that $\mathds 1_{[-k,k]} \sM_c[g]^{-1} \in \IL^2(\IR)$, then we can ensure that $\mathds 1_{[-k,k]}\widehat\sM_c/\sM_c[g] \in \IL^1(\IR) \cap \IL^2(\IR)$ since $|\widehat\sM_c(t)|\leq \widehat \sM_c(0)< \infty$ almost surely. Now, the spectral cut-off density estimator $\widetilde f_k$ can be defined by
\begin{align}\label{eq:sce}
\widetilde f_k(x):=\sM_c^{-1}[\mathds 1_{[-k,k]} \widehat \sM_c/\sM_c[g]](x)= \frac{1}{2\pi} \int_{-k}^k x^{-c-it } \frac{\widehat \sM_c(t)}{\sM_c[g](t)} dt, \quad x\in \IRp.
\end{align}
Here we used two minor assumptions on the error density $g$, that is, 
\begin{align*}
\forall t\in \IR: \sM_c[g](t)\neq 0\quad \text{ and } \quad\forall k\in \IRp: \mathds 1_{[-k,k]}\sM_c[g]^{-1}\in \IL^2(\IR). \tag{\textbf{[G0]}}
\end{align*}
This assumption implies that the Mellin transform of $g$ does not approach zero too fast. Although this assumption is fulfilled for a large class of error densities, we will now show that one can define an estimator for an even weaker assumption on the error density. An intense study of this estimator, including the minimax optimality and data-driven choice of the parameter $k\in \IRp$, can be found in \cite{Brenner-MiguelComteJohannes2020}.
\paragraph{Ridge estimator} Inspired by the work of \cite{Meister2009} and \cite{HallMeister2007}, let $r, \xi \geq 0$ such that $t\mapsto\sM_c[g](t)^{r+1}(1+|t|)^{-\xi(r+2)}\in \IL^1(\IR)\cap \IL^2(\IR).$ Then for any $k\in \IRp$ we define the function $\mathrm R_{k, \xi, r}\in \IL^1(\IR) \cap \IL^2(\IR)$ by
\begin{align*}
\mathrm R_{k,\xi, r}(t):= \frac{\sM_c[g](-t) |\sM_c[g](t)|^r}{\max(|\sM_c[g](t)|, k^{-1}(1+|t|)^{\xi})^{r+2}}, \quad t\in \IR,
\end{align*}
and the set $G_k:=\{t\in \mathbb R: k^{-1}(1+|t|)^{\xi}> |\sM_c[g](t)|\}.$ Now for all $t\in G_n^c$ holds $\mathrm R_{k,r}(t)= \sM_c[g](t)^{-1}$. We define next the \textit{ridge density estimator} $\widehat f_{k, r}$ by $\widehat f_{k,r}:= \sM_c^{-1}[\widehat \sM_c \mathrm R_{k, r}]$. In fact, it can be written explicitly for $x\in \IRp$ as
\begin{align}\label{eq:ridge}
\widehat f_{k,\xi, r}(x)= \frac{1}{2\pi} \int_{-\infty}^{\infty} x^{-c-it} \widehat \sM_c(t) \mathrm R_{k,\xi, r}(t)dt 
= \frac{1}{2\pi} \int_{G_k^c} x^{-c-it} \frac{\widehat\sM_c(t)}{\sM_c[g](t)} dt+ \frac{1}{2\pi} \int_{G_k} x^{-c-it} \widehat\sM_c(t)\mathrm R_{k,\xi, r}(t) dt.
\end{align}
By the construction of $G_k$ the quotient $\widehat\sM_c(t)/\sM_c[g](t)$ in the integrand in eq. \eqref{eq:ridge} is well-defined even without assumption \textbf{[G0]}.

\section{Minimax theory\label{mt}}
In this section, we will see that an even milder assumption on the error density $g$ than \textbf{[G1]} is sufficient to ensure that the presented ridge estimator is consistent. We finish this Section \ref{mt}, by showing that the estimator is minimax optimal over the Mellin-Sobolev ellipsoids. We denote by  $\E_{f_Y}^n$  the expectation corresponding to the distribution of $(Y_j)_{j\in \llbracket n \rrbracket}$. Respectively we define $\E_{f_Y}:=\E_{f_Y}^1$ and $\E_g, \E_f$.  

\subsection{General consistency} Although the sequence $(G_k)_{k\in \mathbb N}$ is obviously nested, that is $G_{k+1}\subseteq G_{k}$ for all $k\in \mathbb N$, we want to stress out that the squared bias, $\|f-\E_{f_Y}^n(\widehat f_{k, \xi, r})\|_{x^{2c-1}}^2$ of $\widehat f_{k, \xi, r}$, defined in  eq. \eqref{eq:ridge}, might not tend to zero for $k$ going to infinity. For instance, one may consider the case where $\sM_c[g]$ vanishes on an open, nonempty set $A\subset \IR$ and $\sM_c[f]$ does not vanish on $A$. A more sophisticated discussion about identifiability and consistency in the context of additive deconvolution problems can be found in the work of \cite{Meister2009}. The discussion there can be directly transfered to the case of multiplicative deconvolution problems. Based on the discussion presented in \cite{Meister2009}, we will give a minimal assumption to ensure that we can define a consistent estimator using the ridge approach. We will from now on assume, that the Mellin transform of $g$ is almost nonzero everywhere, that is,
\begin{align*}
\lambda(\{t\in \IR:  \sM_c[g](t)=0 \})=0.  \tag{\textbf{[G-1]}} 
\end{align*}
Under the asumption \textbf{[G-1]} we can use the dominated convergence theorem to show that the bias $\|f- \E_{f_Y}^n(\widehat f_{k,\xi, r})\|_{x^{2c-1}}^2$ vanishes for $k$ going to infinity.
Further, it is worth stressing out that for $k\in \mathbb N$ and $t\in \mathbb R$ we have $\mathrm R_{k+1,\xi, r}(t) \geq \mathrm R_{k, \xi, r}(t)$.
We then get the following results whose proofs is postponed to  \ref{a:mt}.

\begin{proposition}\label{pro:rid:risk:bound}
	Let $c\in \IR$ such that  $f\in \IL^2(\IRp, x^{2c-1})$ and $\sigma_c:=\E_{f_Y}(Y_1^{2(c-1)})< \infty$. Then for any $r, \xi\geq 0$ with  $\sM_c[g]^{r+1}(1+|t|)^{-\xi(r+2)}\in \IL^1(\IR)\cap \IL^2(\IR)$ we have
	\begin{align*}
	\E_{f_Y}^n(\|f-\widehat f_{k,\xi, r}\|_{x^{2c-1}}^2) \leq \frac{1}{2\pi} \|\mathds 1_{G_k} \sM_c[f]\|_{\IR}^2+ \frac{\sigma_c}{2\pi n} \|\mathrm R_{k,\xi, r}\|_{\IR}^2
	\end{align*}
	where $G_k:=\{t\in \IR: k^{-1}(1+|t|)^{\xi}>|\sM_c[g](t)|\}$ and $\widehat f_{k,\xi, r}$ is defined in equation \eqref{eq:ridge}.\\
	If additionally \textbf{[G-1]} holds and $(k_n)_{n\in\mathbb N}$ satisfies $k_n\rightarrow \infty$ and $n^{-1} \|\mathrm R_{k_n,\xi, r}\|_{\IR}^2 \rightarrow 0$ for $n\rightarrow \infty$ then
	\begin{align*}
	\E_{f_Y}^n(\|\widehat f_{k_n, \xi, r}-f\|_{x^{2c-1}}^2) \rightarrow 0
	\end{align*}
	for $n\rightarrow \infty.$
\end{proposition}

Although the assumptions on $\xi, k_n, r\geq 0$ in Proposition \ref{pro:rid:risk:bound} seem to be rather technical, we will see that they are fulfilled when considering more precise classes of error densities, so-called \textit{smooth error densities}. Before we define this family of error densities let us shorty comment on the consistency of the presented estimator.
\begin{remark}[Strong consistency]\label{rem:st:con}
	In Proposition \ref{pro:rid:risk:bound} we have seen that we can determine a set of assumptions which ensures by application of the Markov inequality,  that $\|\widehat f_{k_n,\xi,r }-f\|_{x^{2c-1}}^2 \rightarrow 0$ in probability. Here, we needed the additional assumption that $f\in \IL^2(\IR, x^{2c-1})$ and $\sigma_c=\E_{f_Y}(Y_1^{2(c-1)})<\infty$ to construct the estimator and show its properties. A less restrictive metric which can be considered would be the $\IL^1(\IRp, x^0)$-metric, since for any density, $f\in \IL^1(\IRp, x^{0})$ holds. Further, the Mellin transform developed in $c=1$ is well-defined for any density $f$. In the book of \cite{Meister2009} they proposed an estimator $\widehat f_V$ of the density $f_V:\IR \rightarrow \IR$ of a real random variable $V$ given i.i.d. copies of $Z=V+\varepsilon$ where $V$ and $\varepsilon$ are stochastically independent. They were able to show that their estimator $\widehat f$ is strongly consistent in the $\IL^1(\IR)$-sense, that is, $\|\widehat f_V- f_V\|_{\IL^1(\IR)} \rightarrow 0$ almost surely. Given the $\log$ transformed data, $\log(Y)= \log(X)+\log(\varepsilon)$, we can use the estimator $\widehat f_V$ for $V=\log(X)$ and deduce the estimator $\widehat f_X(x):= \widehat f_V(\log(x))x^{-1}$ for any $x\in \IRp.$ Then $\|\widehat f_X-f\|_{\IL^1(\IRp, x^0)}= \|\widehat f_V-f_V\|_{\IL^1 (\IR)}$, implying that the estimator $\widehat f_X$ is strongly consistent in the $\IL^1(\IRp, x^0)$. Although it might be tempting generalise this result for the $\IL^1(\IRp, x^{c-1})$-distance for any $c\in \IR$, it would need an additional moment assumption on $f$ which contradicts the idea of considering the most general case.
\end{remark}

\subsection{Noise assumption} Up to now, we have only assumed that the Mellin transform of the error density $g$ does not vanish almost everywhere, i.e. \textbf{[G-1]}. To develop the minimax theory for the estimator $\widehat f_{k,\xi, r}$ we will specify the class of considered error density $g$ through an assumption on the decay of its corresponding Mellin transform $\sM_c[g]$. This assumption will allow us to determine the growth of the variance term more precisely. In the context of additive deconvolution problems, compare \cite{Fan1991}, densities whose Fourier transform decay polynomially are called \highlight{smooth error densities}. To stay in this way of speaking we say that an error density $g$ is a \highlight{smooth error density} if there exists $c_g,C_g, \gamma \in \IRp$ such that
\begin{align*}
\tag{\textbf{[G1]}}  c_g  (1+t^2)^{-\gamma/2} \leq |\sM_c[g](t)| \leq C_g (1+t^2)^{-\gamma/2},\quad  t\in \mathbb R.
\end{align*} 
This assumption on the error density  was also considered in the works of \cite{BelomestnyGoldenshluger2020}, \cite{Brenner-MiguelComteJohannes2020} and \cite{Brenner-Miguel2021}.
We focus on to the case where $\xi=0$, and use the abreviation $\widehat f_{k}:= \widehat f_{k, 0, r}$, respectively $\mathrm R_{k}:= \mathrm R_{k, 0, r}$. Then under the asummption of Proposition \ref{pro:rid:risk:bound} and assumption \textbf{[G1]} we can show that for each  $r>0\vee (\gamma^{-1}-1)$ there exists a constant $C_{g,r}>0$ such that $n^{-1} \|\mathrm R_{k}\|_{\IR}^2\leq C_{g,r} k^{2+ \gamma^{-1}}n^{-1}$, which leads to the following corollary whose proof can be found in \ref{a:mt}. Here $a\vee b:=\max(a,b)$ for $a,b\in \IR.$

\begin{corollary}\label{cor:rid:consis:ex}
	Let the assumptions of Proposition \ref{pro:rid:risk:bound} and \textbf{[G1]} be fulfilled. Then for $r>0\vee(\gamma^{-1}-1)$, 
	\begin{align*}
	\E_{f_Y}^n(\|f-\widehat f_{k}\|_{x^{2c-1}}^2 ) \leq \frac{1}{2\pi} \|\mathds 1_{G_k} \sM_c[f]\|_{\IR}^2 + C_{g,r} \frac{\sigma_ck^{2+\gamma^{-1}}}{n} .
	\end{align*}
	If one chooses $(k_n)_{n\in \mathbb N}$ such that $k_n^{2+\gamma^{-1}} n^{-1} \rightarrow 0$ and $k_n \rightarrow \infty$  then $\E_{f_Y}^n(\|f-\widehat f_{k_n}\|_{x^{2c-1}}^2 ) \rightarrow 0$ for $n\rightarrow \infty.$
	\end{corollary}
Considering the bound of the variance term, a  choice of $(k_n)_{n\in \mathbb N}$ increasing slowly in $n$, would imply a faster decay of the variance term. On the other hand, the opposite effect on the bias term can be observed. In fact, to balance both terms,  an assumption on the decay of the Mellin transform of the unknown density $f$ is needed. In the non-parametric Statistics and in the inverse problem community this is usually done by considering regularity spaces.
\subsection{The Mellin-Sobolev space} We will now introduce the so-called \highlight{Mellin-Sobolev} spaces, which are, for instance,  considered by \cite{Brenner-MiguelComteJohannes2020} for the case $c=1$ and \cite{Brenner-Miguel2021} for the multivariate case. In the work of \cite{Brenner-MiguelComteJohannes2020} their connection to regularity properties, in terms of analytical properties, and their connection to the Fourier-Sobolev spaces are intensely studied.
For $s, L\in \IRp$ and $c\in \IR$ we define the \highlight{Mellin-Sobolev spaces} by
\begin{align*}
\mathbb W^{s}_c(\IRp):=\{ f\in\IL^2(\IRp, x^{2c-1}): |f|_{s,c}^2:= \|(1+t^2)^{s/2} \sM_c[f]\|_{\IR}^2 \leq \infty\}
\end{align*} and their corresponding ellipsoids by $\mathbb W^s_c(L):=\{f\in \mathbb W^s_c(\IRp): |f|_{s,c}^2 \leq L\}$. Then for $f\in \mathbb W^s_c(L)$ and under assumption \textbf{[G1]} we can show that  $\|\mathds 1_{G_k} \sM_c[f]\|_{\IR}^2 \leq C(g, L, s) k^{-2s/\gamma}$.
Since $f$ is a density and to control the variance term, it is natural to consider the following subset of $\mathbb W^s_c(L),$
\begin{align*}
\mathbb D^s_c(L):=\{ f\in \mathbb W_c^s(L): f \text{ is a density }, \E_f(X_1^{2(c-1)})\leq L\}.
\end{align*}
Then we can state the following theorem whose proof is postponed to \ref{a:mt}.
\begin{theorem}[Upper bound over $\mathbb D_c^s(L)$]\label{thm:upp:bound}
	Let $c\in \IR$, $s,L\in \IRp$ and  $\E_{g}(U_1^{2(c-1)})< \infty$. Let further \textbf{[G1]} be fulfilled and $r>0\vee(\gamma^{-1}-1)$. Then the choice $k_o:=n^{\gamma/(2s+2\gamma+1)}$ leads to
	\begin{align*}
	\sup_{f\in \mathbb D_c^s(L)} \E^n_{f_Y}(\|f-\widehat f_{k_o}\|_{x^{2c-1}}^2) \leq C_{g,  L, r} n^{-2s/(2s+2\gamma+1)}.
	\end{align*}
\end{theorem}
A similar result was presented by the authors \cite{Brenner-MiguelComteJohannes2020} and \cite{Brenner-Miguel2021} showing that for the spectral cut-off estimator $\widetilde f_{k_o}$ the choice $k_o=n^{1/(2s+2\gamma+1)}$ leads to the same rate of $n^{-2s/(2s+\gamma+1)}$ uniformly over the classes $\mathbb D_c^s(L)$.
For the case $c=1$ the authors of \cite{Brenner-MiguelComteJohannes2020}  presented a lower bound result, showing that in many cases the rate given in Theorem \ref{thm:upp:bound} is the minimax rate for the density estimation $f$ given the i.i.d. sample $(Y_j)_{j\in \llbracket n \rrbracket}$. For the multivariate case, the author in \cite{Brenner-Miguel2021} has generalised the proof for all $c>0$. The following Theorem follows as a special case of the lower bound presented in \cite{Brenner-Miguel2021} for the dimension $d=1$ and its proof is thus omitted.
\begin{theorem}[Lower bound over $\mathbb D_c^s(L)$]\label{thm:low:bound}
	Let $s, \gamma \in \mathbb N$, $c>0$ and assume that \textbf{[G1]} holds. Additionally, assume that $g(x)=0$ for $x>1$ and that there exists constants $c_g, C_g$ such that
	\begin{align*}
		c_g(1+t^2)^{-\gamma/2}\leq|\sM_{\widetilde c}[g](t)| \leq C_g (1+t^2)^{-\gamma/2}, \quad t\in \IR,
	\end{align*}
	where $\widetilde c=1/2$ for $c> 1/2$ and $\widetilde c=0$ for $c\in (0,1/2]$. \\
	Then there exist constants $C_{g,c}, L_{s, g, c}>0$ such that for all $L\geq L_{s,g,c}, n\in \mathbb N$ and for any estimator $\widehat f$ of $f$ based on an i.i.d. sample $(Y_j)_{j\in \nset{n}}$,
	\begin{align*}
	\sup_{f\in \mathbb D_c^s(L)} \E_{f_Y}^n(\|\widehat f-f\|_{x^{2c-1}}^2) \geq C_{g,c} n^{-2s/(2s+2\gamma+1)}.
	\end{align*}
	\end{theorem}
We want to emphasize that the additional assumption on the error densities are for technical reasons. To ensure that $\sM_{\widetilde c}[g]$ is well-defined, we need to addtionally assume that $\E_g(U_1^{-1/2})< \infty$ for the case of $c>1/2$. If $c\in (0,1/2],$ then $\E_g(U_1^{-1})<\infty$ follows from $\E_g(U_1^{2c-2})$, compare Proposition \ref{pro:rid:risk:bound}. \\
In the work of \cite{Brenner-MiguelComteJohannes2020} the authors showed that for the case of $c=1$ the spectral cut-off estimator $\widetilde f_k$, defined in eq. \eqref{eq:sce} is minimax optimal for some examples of error densities. In fact, they stressed out that for Beta-distributed $U_1$, considered for instance by \cite{BelomestnyComteGenon-Catalot2016}, all assumption on $g$ are fulfilled. 

\section{Data-driven method\label{dd}}
In Section \ref{mt} we determined a choice of the parameter $\xi, k, r\geq 0$ such that the resulting ridge estimator $\widehat f_{k,\xi, r}$ is consistent, see Corollary \ref{cor:rid:consis:ex}. Setting $\xi=0$ we additionally found a choice of the parameter $k\in \IRp$ which makes the estimator minimax optimal over the Mellin-Sobolev ellipsoids $\mathbb D_c^s(L)$, compare Theorem \ref{thm:upp:bound}. We want to emphasize that the latter choice of $k\in \IRp$ might not be explicitly dependent on the exact unknown density $f$ but is still dependent on its regularity parameter $s\in \IRp$ which again  is unknown. \\
We will now present a data-driven version of the estimator $\widehat f_{k,r}$ only dependent on the sample $(Y_j)_{j\in \llbracket n \rrbracket}$. 
For the data-driven choice of $k\in \IRp$ we will use a version of the Goldenshluger-Lepski method. That is, we will define the random functions $\widehat A, \widehat V: \IRp\rightarrow \IRp$ for $k\in \IRp$ by
\begin{align*}
 \widehat A(k):= \sup_{k'\in\mathcal K_n} (\|\widehat f_{k'}- \widehat f_{k'\wedge k}\|_{x^{2c-1}}^2- \chi_1\widehat V(k))_+ \text{ and } \widehat V(k):= 2\widehat \sigma_c \|\mathrm R_k\|_{\IR}^2n^{-1}
\end{align*}
for $ \chi_1\in \mathbb R_+$ and $\mathcal K_n:=\{k\in \mathbb N: \|\mathrm R_k\|_{\IR}^2\leq n\}$. Here $a\wedge b:=\min(a,b)$ and $a_+=\max(a, 0)$ for any real numbers $a,b\in \IR$. 
Generally, the random function $\widehat V$ is an empiricial version of   $V(k):=\sigma_c\|\mathrm R_{k}\|_{\IR}^2n^{-1}$ which mimics the behaviour of the variance term, compare Proposition \ref{pro:rid:risk:bound}. Analogously, $\widehat A$ is an empirical version of $A(k):=\sup_{k'\in\mathcal K_n} (\|\widehat f_{k'}- \widehat f_{k'\wedge k}\|_{x^{2c-1}}^2- \chi_1V(k))_+$ which behaves like the bias term. 
For $\chi_2 \geq \chi_1$ we then set \begin{align}\label{gde:mme:lm}
\widehat k:= \argmin_{k\in \mathcal K_n} \widehat A(k)+ \chi_2\widehat V(k).
\end{align}
Then we can show the following result where we denote by $\|h\|_{\infty}$ the essential supremum of a measurable function $h:\IR\rightarrow \mathbb C$ and $\|h\|_{\infty,x^{2c-1}}$ the essential supremum of $x\mapsto x^{2c-1}h(x)$. 
\begin{theorem}\label{thm:dd:ridge}
	Let $c\in \IR$ and $f\in \IL^2(\IRp, x^{2c-1})$. Assume that $\E_{f_Y}(Y_1^{5(c-1)})< \infty$, $\|g \|_{\infty, x^{2c-1}}<\infty$ and \textbf{[G1]} is fulfilled. Then for $\chi_2 \geq \chi_1 \geq 72$,
	\begin{align*}
	\E_{f_Y}^n(\|\widehat f_{\widehat k}- f\|_{x^{2c-1}}^2) \leq C_1 \inf_{k\in \mathcal K_n} \left( \|\mathds 1_{G_k} \sM_c[f]\|_{\IR}^2 + V(k) \right) + \frac{C_2}{n} 
	\end{align*}
	where $C_1$ is a positive constant depending on $\chi_2, \chi_1$ and  $C_2$ is a positive constant depending on $\E_{f_Y}(Y_1^{5(c-1)}), \|g\|_{\infty, x^{2c-1}}$, $g$ and $r$.
\end{theorem}
	Assuming now that the density lies in a Mellin-Sobolev ellipsoid, we can deduce directly the following corollary whose proof is thus omitted.
	\begin{corollary}
		Let $c\in \IR$, $s,L\in \IRp$ and $f\in \mathbb D^s_c(L)$. Assume further that  $\E_{f_Y}(Y_1^{5(c-1)})< \infty$, $\|g \|_{\infty,x^{2c-1}}<\infty$ and \textbf{[G1]} is fulfilled.   Then for $\chi_2 \geq \chi_1 \geq 72$,
		\begin{align*}
		\E_{f_Y}^n(\|\widehat f_{\widehat k}-f\|_{x^{2c-1}}^2) \leq C(L, s, r, g, \E_f(X_1^{5(c-1)})) \, n^{-2s/(2s+2\gamma+1)}
		\end{align*}
		where $C(L, s, r, g, \E_f(X_1^{5(c-1)}))$ is a positive constant depending on $L, s, r, g$ and $\E_f(X_1^{5(c-1)})$.
	\end{corollary}
\section*{Conclusion}
Let us summarise the presented results of the ridge estimator $\widehat  f_{k,\xi, r}$ in comparison to the properties of the spectral cut-off estimator $\widetilde f_k$, considered by \cite{Brenner-MiguelComteJohannes2020} and \cite{Brenner-Miguel2021}. For the definition of the estimator, the spectral cut-off estimator needs the assumption \textbf{[G0]}. This assumption already implies the existence of a consistent version of the spectral cut-off estimator. For the definition of the ridge estimator the assumption \textbf{[G0]} is not necessary. Nevertheless, in order to show that there exists a consistent version of the ridge estimator, we needed assumption \textbf{[G-1]}, which is weaker than \textbf{[G0]}. In this scenario, the estimator $\widehat f_{k,\xi,r}$ seems to be favourable if one aims to consider minimal assumptions on the error density, for instance to construct a strong consistent estimator, compare Remark \ref{rem:st:con}. As soon as we are interested in developing the minimax theory of the estimators, the assumption \textbf{[G1]} is natural to be considered. It is worth pointing out, that \textbf{[G1]} implies \textbf{[G0]} and therefore \textbf{[G-1]}. Here the assumptions of Proposition \ref{pro:rid:risk:bound}, which are needed for  the minimax optimality of both estimators, are identical to the assumptions of \cite{Brenner-Miguel2021}. Thus none of the estimators seem to be more favourable in terms of minimax-optimality. Again, for the data-driven estimators $\widehat f_{\widehat k}$ and  $\widetilde f_{\widetilde k}$, proposed by \cite{Brenner-MiguelComteJohannes2020}, the assumptions on the error densities are identical. Here it should be mentioned that the authors \cite{Brenner-MiguelComteJohannes2020} have proven the case $c=1$. The general case for $c\in \IR$ can be easily shown using the same strategies as in the proof of Theorem \ref{thm:dd:ridge}. In total, we can say that for the construction of an estimator with minimal assumption on the error density $g$, the ridge estimator seems to be favourable, in the sense, that it  requires weaker assumptions on $g$. As soon as we consider smooth error densities, that is under assumption \textbf{[G1]}, neither the ridge estimator nor the spectral cut-off estimator seems to be more favourable in terms of minimax-optimality and data-driven estimation.

\section{Numerical study\label{si}}
	In this section, we illustrate the behaviour of the data-driven ridge estimator $\widehat f_{\widehat k}=\widehat f_{\widehat k, 0, r}$ presented in eq. \eqref{eq:ridge} and \eqref{gde:mme:lm} and compare it with the spectral cut-off estimator $\widetilde f_{\widetilde k}$, presented in \cite{Brenner-MiguelComteJohannes2020}, where
	\begin{align*}
	\widetilde k=\argmin_{k\in \widetilde{\mathcal K}_n} -\|\widetilde f_{k}\|_{x^{2c-1}}^2 + \widehat{\mathrm{pen}}(k)
	\end{align*}
	with $\widehat{\mathrm{pen}}(k):= 2\chi\widehat \sigma_c \|\mathds 1_{[-k,k]} \sM_c[g]^{-1}\|_{\IR}^2/(2\pi n)$ and $\widetilde{\mathcal K}_n:=\{k \in \mathbb N: \|\mathds 1_{[-k,k]} \sM_c[g]^{-1}\|_{\IR}^2\leq 2\pi n\}$. To do so, we use the following examples for the unknown density $f$,
\begin{enumerate}
	\item[$(i)$] \highlight{Beta Distribution:} $f(x)= B(2,5)^{-1}x (1-x)^4 \mathds 1_{(0,1)}(x), x\in \IRp$, 
	\item[$(ii)$] \highlight{Log-Gamma Distirbution:} $f(x)= 5^5\Gamma(5)^{-1}x^{-6} \log(x)^4 \mathds 1_{(1,\infty)}(x), x\in \IRp$, 
	\item[$(iii)$] \highlight{Gamma Distribution:} $f(x)= \Gamma(5)^{-1}x^4 \exp(-x) \mathds 1_{(0,\infty)}(x), x\in \IRp$, and 
	\item[$(iv)$] \highlight{Log-Normal Distiribution:} $f(x)= (0.08\pi x^2)^{-1/2}\exp(-\log(x)^2/0.08)\mathds 1_{(0,\infty)}(x), x\in \IRp$.
\end{enumerate} 
A detailed discussion of these examples in terms of the decay of their Mellin transform can be found in \cite{Brenner-Miguel2021}. To visualize the behaviour of the estimator, we use the following examples of error densities $g$,
\begin{enumerate}
	\item[$a)$] \highlight{Symmetric noise:} $g(x)= \mathds 1_{(0.5, 1.5)}(x), x\in \IRp$, and
	\item[$b)$] \highlight{Beta Distribution:} $g(x)= 2x\mathds 1_{(0,1)}(x), x\in \IRp$.
\end{enumerate}
Here it is worth pointing out that the example $a)$ and $b)$ fulfill \textbf{[G1]} with $\gamma=1$ and $\gamma =2$.  By minimising an integrated weighted squared error over a family of histogram
densities with randomly  drawn partitions  and weights we select for $a)$ $\chi_1=\chi_2=72$ for $\widehat f_{\widehat k}$ and $\chi = 5$ for $\widetilde f_{\widetilde k}$. For the case $b)$ we choose $\chi_1=\chi_2=6$ and $\chi=3$. In both cases, we have set $r=2$.

\begin{minipage}{\textwidth}
	\centering{\begin{minipage}[t]{0.35\textwidth}
			\includegraphics[width=\textwidth,height=50mm]{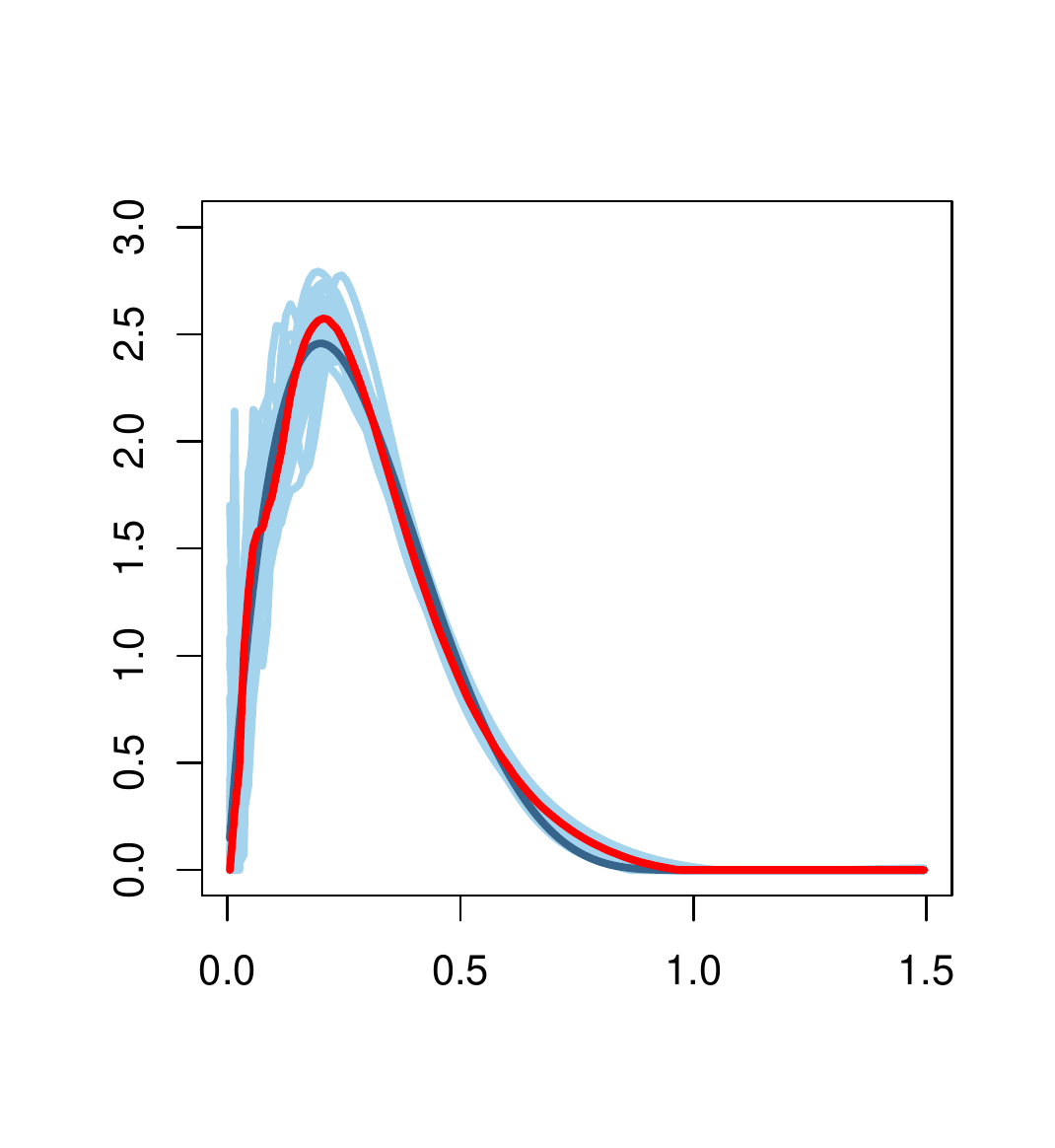}
		\end{minipage}
		\begin{minipage}[t]{0.35\textwidth}
			\includegraphics[width=\textwidth,height=50mm]{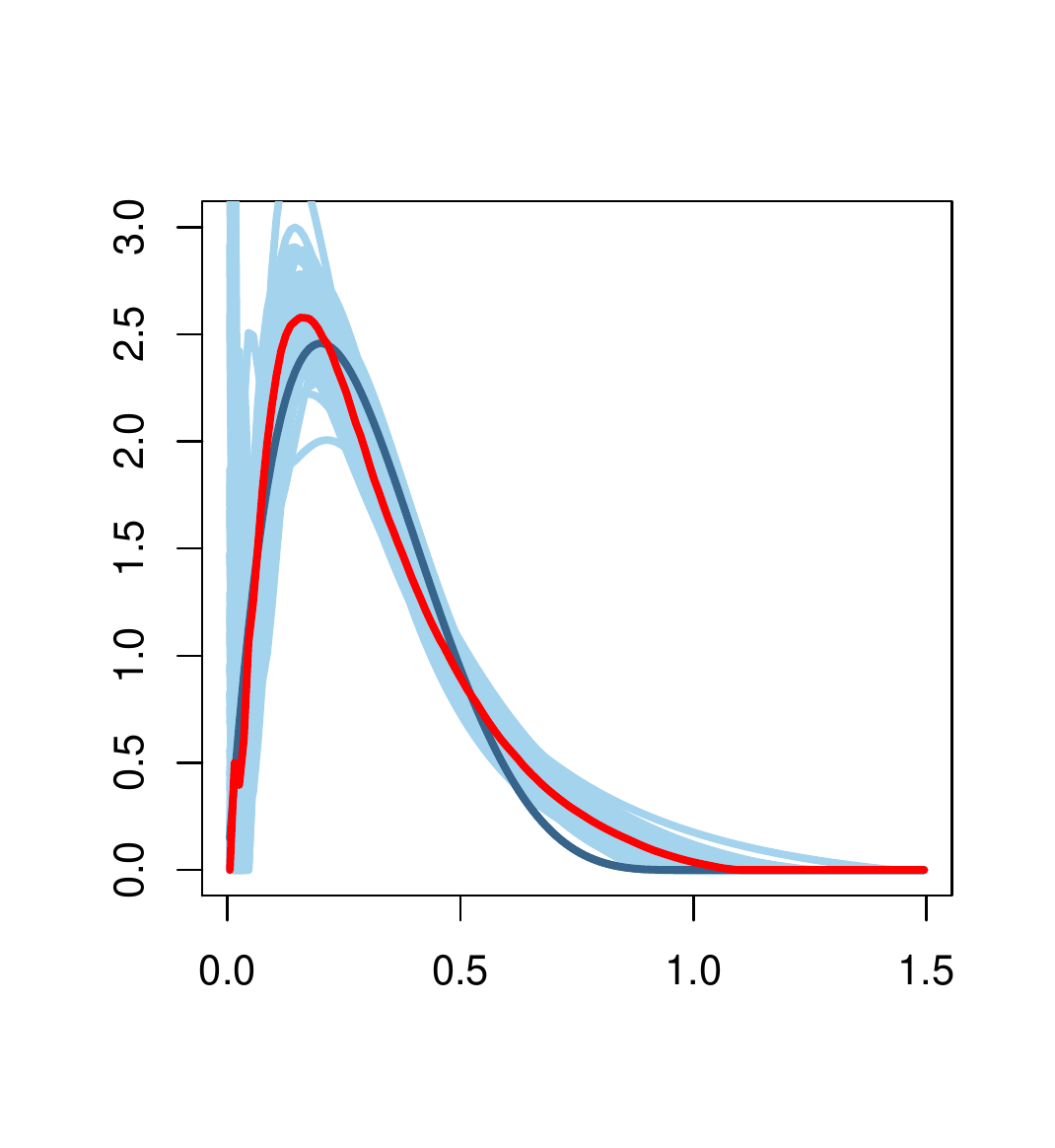}
	\end{minipage}}
	\centering{\begin{minipage}[t]{0.35\textwidth}
			\includegraphics[width=\textwidth,height=50mm]{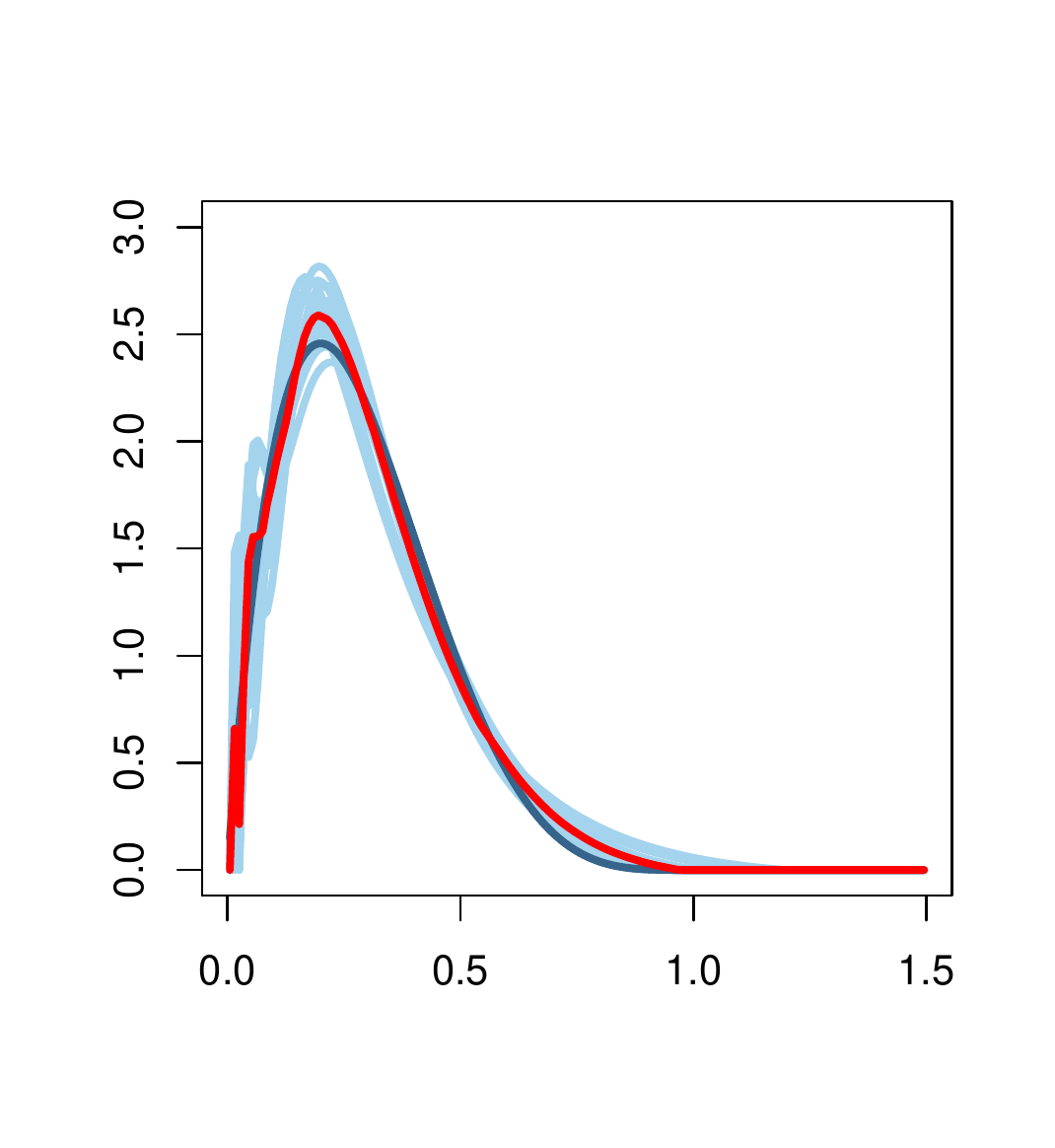}
		\end{minipage}
		\begin{minipage}[t]{0.35\textwidth}
			\includegraphics[width=\textwidth,height=50mm]{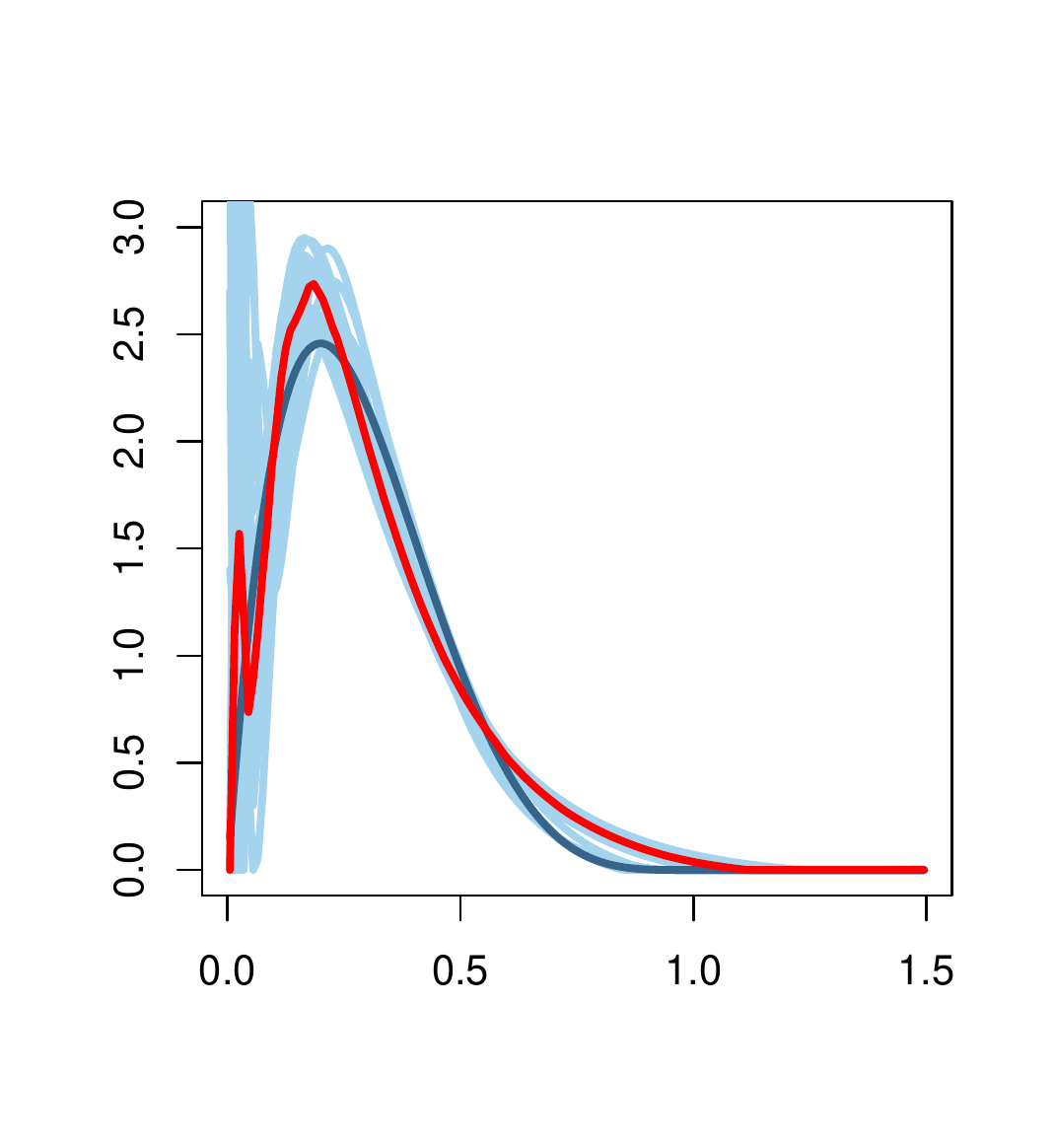}
	\end{minipage}}
	
	\captionof{figure}{\label{figure:1}The estimator $\widehat f_{\widehat k}$ (top) and $\widetilde f_{\widetilde k}$ (bottom) is depicted for 
		50  Monte-Carlo simulations with  sample size  $n=2000$  in the case $(i)$ under the error density $a)$ (left) and $b)$ (right) for $c=1$. The true density $\So$ is given by the black curve while the red  curve is the point-wise empirical median of the 50 estimates.}
\end{minipage}

Figure \ref{figure:1} shows that both estimators behave similarly. As suggested by the theory, the reconstruction of the density $f$ from the observation $(Y_j)_{j\in \llbracket n \rrbracket}$ seems to be less difficult if the error variable is uniformly distibuted, case $a)$, than if the error variable is Beta distributed, case $b)$.

\begin{minipage}{\textwidth}
	\centering{\begin{minipage}[t]{0.32\textwidth}
			\includegraphics[width=\textwidth,height=50mm]{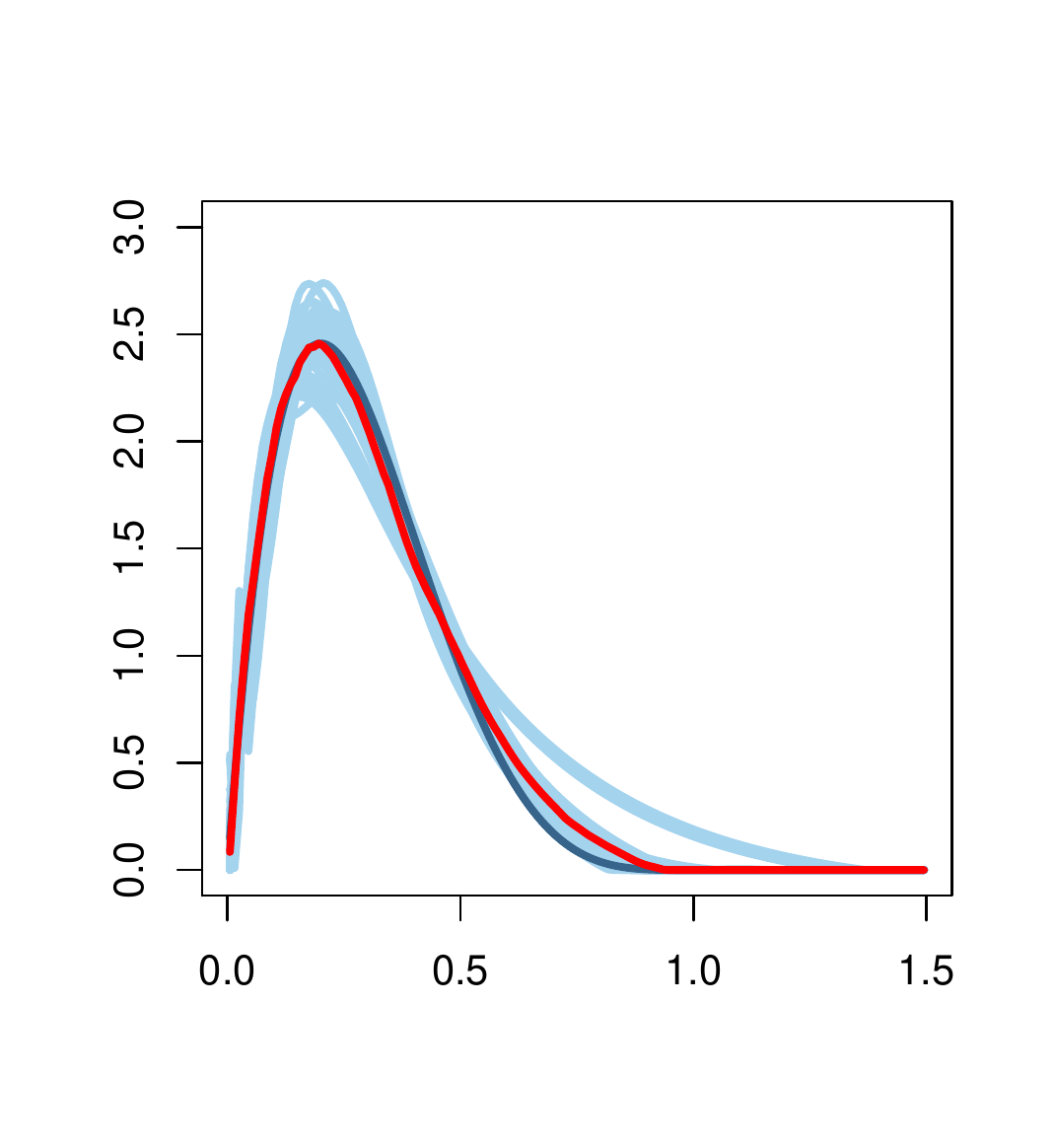}
		\end{minipage}
	\begin{minipage}[t]{0.32\textwidth}
		\includegraphics[width=\textwidth,height=50mm]{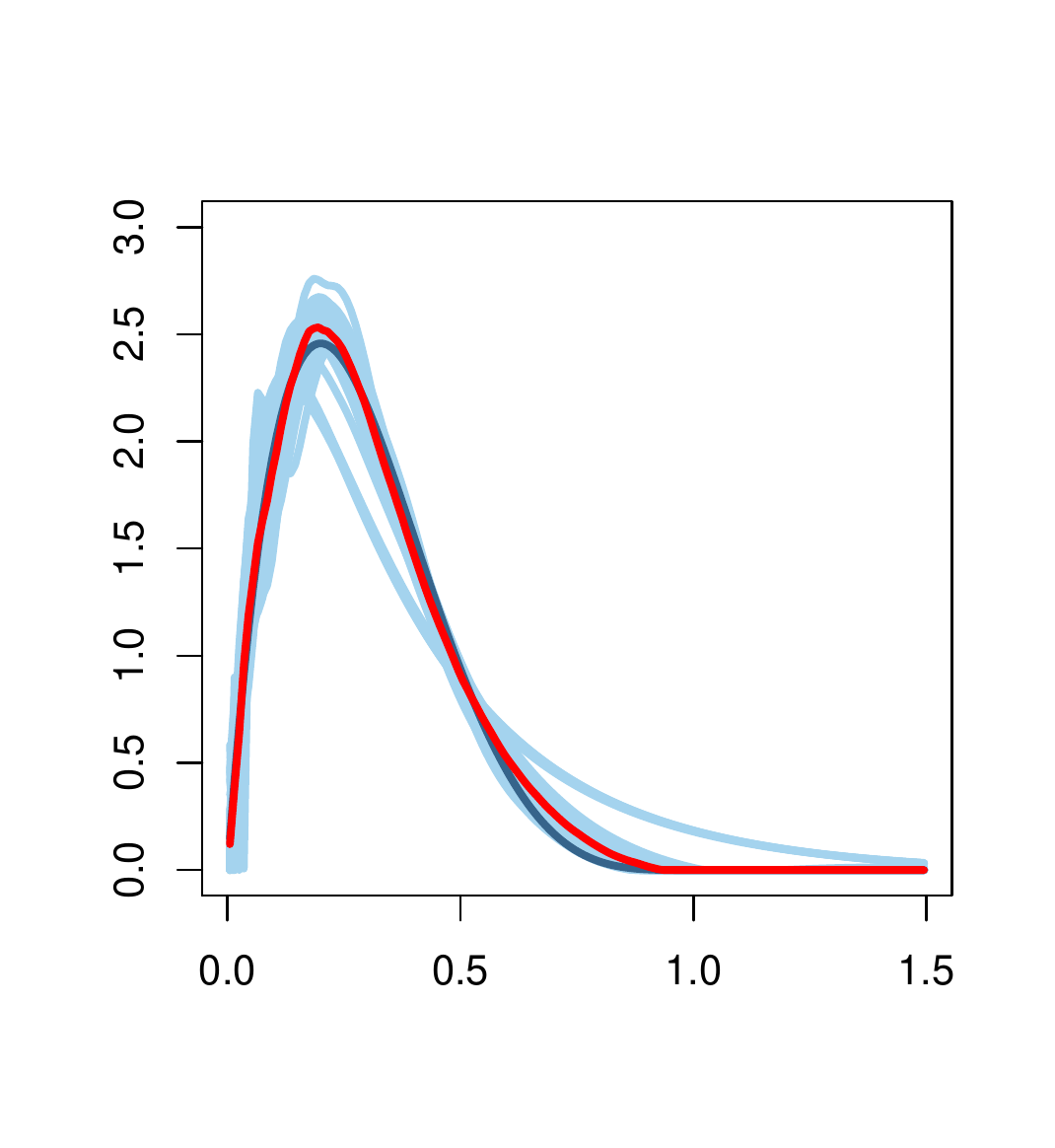}
	\end{minipage}
		\begin{minipage}[t]{0.32\textwidth}
			\includegraphics[width=\textwidth,height=50mm]{pics/beta_2000_ridge_1_sym.pdf}
	\end{minipage}}
	\centering{\begin{minipage}[t]{0.32\textwidth}
			\includegraphics[width=\textwidth,height=50mm]{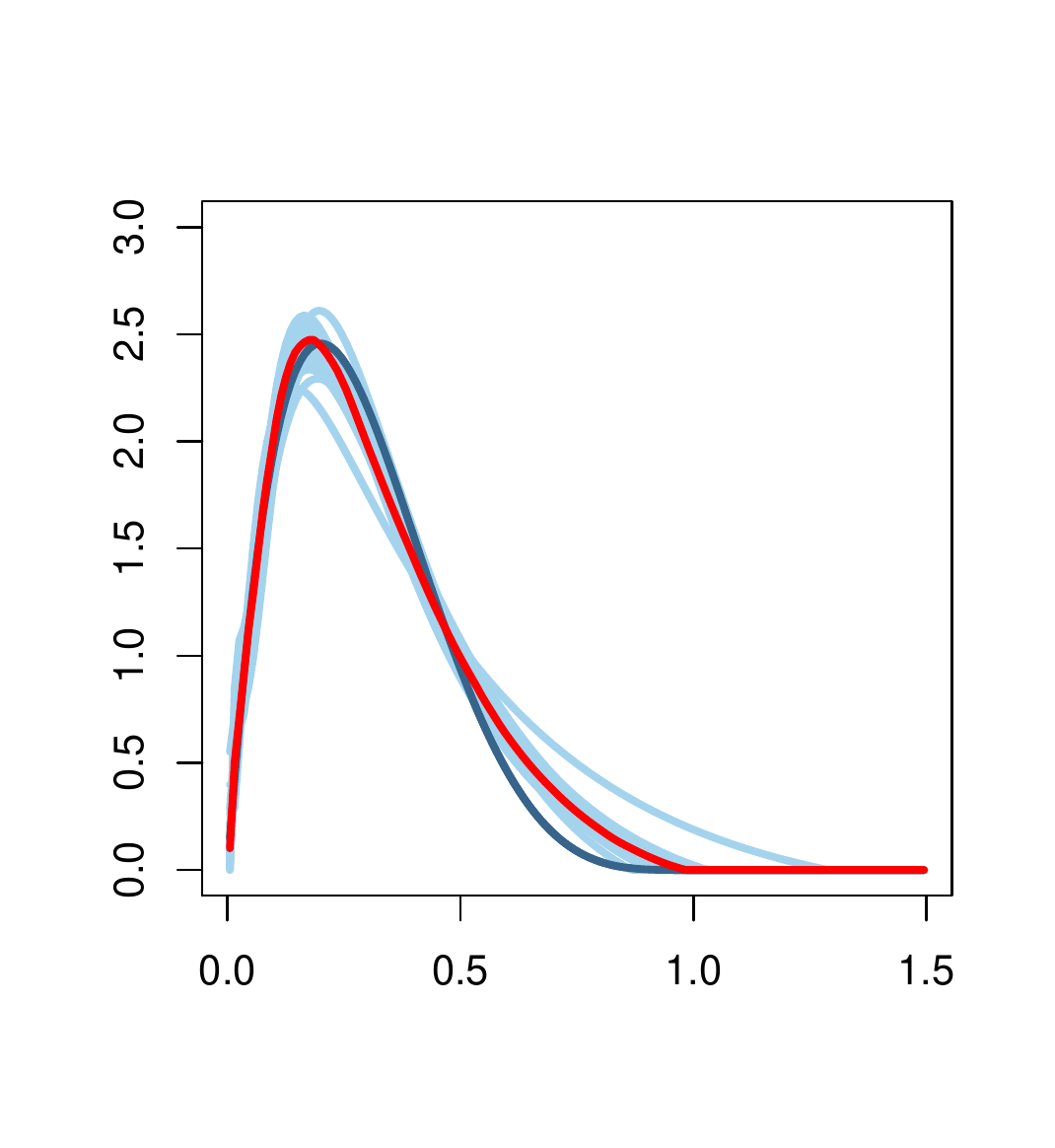}
		\end{minipage}
	\begin{minipage}[t]{0.32\textwidth}
		\includegraphics[width=\textwidth,height=50mm]{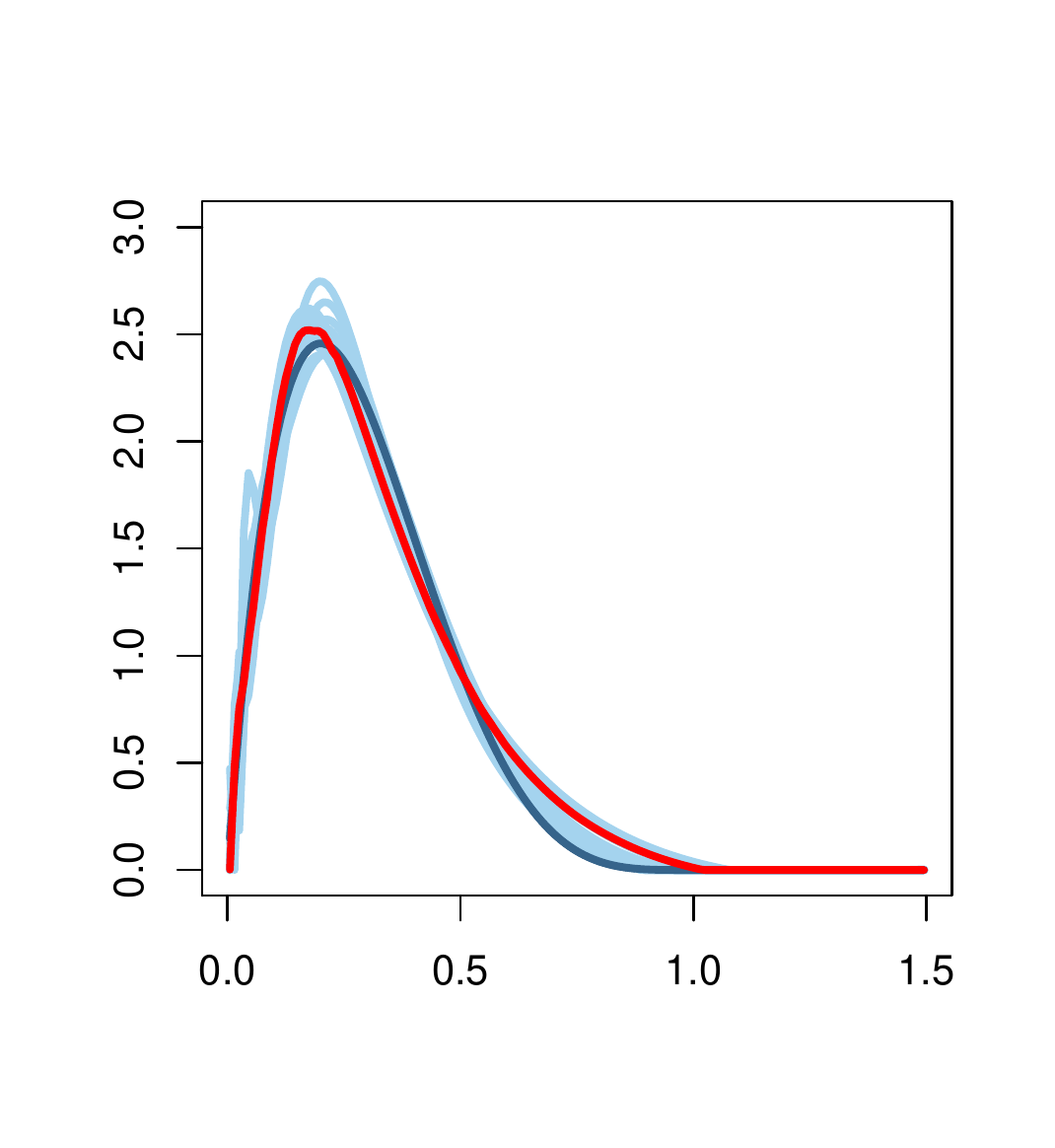}
	\end{minipage}
		\begin{minipage}[t]{0.32\textwidth}
			\includegraphics[width=\textwidth,height=50mm]{pics/beta_2000_sc_1_sym.pdf}
	\end{minipage}}
	
	\captionof{figure}{\label{figure:2}The estimator $\widehat f_{\widehat k}$ (top) and $\widetilde f_{\widetilde k}$ (bottom) is depicted for 
		50  Monte-Carlo simulations with  sample size  $n=2000$  in the case $(i)$ under the error density $a)$ for $c=0$ (left), $c=1/2$ (middle) and $c=1$ (right). The true density $\So$ is given by the black curve while the red  curve is the point-wise empirical median of the 50 estimates.}
\end{minipage}
Again we see that both estimators react analogously to varying values of the model parameter $c\in \IR$. Looking at the medians in Figure \ref{figure:2}, for $c=0$ the median seems to be closer to the true density for smaller values of $x\in \pRz$. For $c=1$ the opposite effects seems to occur. For $c=1/2$, the case of the unweighted $\mathbb L^2$-distance, such effects cannot be observed. Regarding the risk, this seems natural as the weight function for $c=0$ is montonically decreasing, while for $c=1$ it is monotonically increasing.

\begin{table}[ht]
	\centerline{\begin{tabular}{@{}cc||c c|cc|cc|cc@{}}
			&Case& $(i)$ & & $(ii)$ & & $(iii)$&&$(iv)$&\\
			&Sample size& $500$& $2000$ & $500$ & $2000$ & $500$ & $2000$ &$500$&$2000$\\\midrule\midrule
			$a)$&Ridge & $0.94$& $0.31$ & $2.17$ & $1.54$ & $0.63$ & $0.17$ &$7.13$&$2.38$\\
			&Spectral & $1.10$& $0.38$ & $2.03$ & $1.26$ & $0.52$ & $0.16$ &$15.07$&$2.34$ \\
			$b)$&Ridge & $2.32$& $1.43$ & $5.90$ & $3.81$ & $1.19$ & $0.47$ &$25.84$&$11.03$\\
			&Spectral & $3.95$& $1.56$ & $10.63$ & $7.12$ & $1.52$ & $0.84$ &$33.95$&$13.45$
			 \\\bottomrule
	\end{tabular}}
	\caption{The entries showcase the MISE (scaled by a factor of 100) obtained by Monte-Carlo simulations each with 500 iterations. We take a look at different densities $f$ and $g$, two distinct sample sizes and for both estimators $\widehat f_{\widehat k}$ and $\widetilde f_{\widetilde k}$ we set $c=1$.}
	\label{table:num_res}
\end{table}

\appendix

\section{Preliminary}\label{a:prel}
We will start by defining the Mellin transform for square-integrable functions $h \in \IL^2(\IRp, x^{2c-1})$ and collect some of its major properties. Proof sketches for all the mentioned results can be found in \cite{Brenner-MiguelComteJohannes2020}, respectively \cite{Brenner-Miguel2021}.

\paragraph{The Mellin transform}
To define the Mellin transform of a  square-integrable function, that is for $h_1\in \IL^2(\IR_+, x^{2c-1})$, we make use of the definition of the Fourier-Plancherel transform. To do so, let $\varphi:\IR \rightarrow \IR_+, x\mapsto \exp(-2\pi x)$ and $\varphi^{-1}: \IR_+ \rightarrow \IR$ be its inverse. Then, as diffeomorphisms, $\varphi, \varphi^{-1}$ map  Lebesgue null sets on Lebesgue null sets. Thus the isomorphism $\Phi_c:\IL^2(\IR_+,x^{2c-1}) \rightarrow \IL^2(\IR), h\mapsto \varphi^c \cdot(h\circ \varphi)$ is well-defined. Moreover, let $\Phi^{-1}_c: \IL^2(\IR) \rightarrow \IL^2(\IR_+,x^{2c-1})$ be its inverse. Then for $h\in \IL^2(\IR_+,x^{2c-1})$ we define the Mellin transform of $h$ developed in $c\in \IR$ by
\begin{align*}\label{eq:mel;def}
\sM_c[h](t):= (2\pi) \mathcal F[\Phi_c[h]](t), \quad t\in \IR,
\end{align*} 
where $\mathcal F: \IL^2(\IR)\rightarrow \IL^2(\IR), H\mapsto (t\mapsto \mathcal F[H](t):=\lim_{k\rightarrow \infty}\int_{-k}^k \exp(-2\pi i t x) H(x) dt)$ is the Plancherel-Fourier transform. Due to this definition several properties of the Mellin transform can be deduced from the well-known theory of Fourier transforms. In the case that $h\in \IL^1(\IR_+,x^{c-1}) \cap \IL^2(\IR_+,x^{2c-1})$ we have 

\begin{align}
\sM_c[h](t) =\int_0^{\infty} x^{c-1+it} h(x)dx, \quad  t\in \IR
\end{align}
which coincides with the usual notion of Mellin transforms as considered in \cite{ParisKaminski2001}. 

Now, due to the construction of the operator $\mathcal M_c: \IL^2(\IR_+,x^{2c-1}) \rightarrow \IL^2(\IR)$ it can easily be seen that it is an isomorphism. We denote by $\mathcal M_c^{-1}: \IL^2(\IR) \rightarrow \IL^2(\IR_+,x^{2c-1})$ its inverse. If additionally to $H\in \IL^2(\IR)$, $H\in \IL^1(\IR)$, we can express the inverse Mellin transform explicitly through
\begin{align}\label{eq:Mel:inv}
\sM_{c}^{-1}[H](x)= \frac{1}{2\pi } \int_{-\infty}^{\infty} x^{-c-it} H(t) dt, \quad  x\in \IR_+.
\end{align} 
Furthermore, we can directly show that a Plancherel-type equation holds for the Mellin transform, that is for all $h_1, h_2 \in \IL(\IR_+,x^{2c-1})$, 
\begin{align}\label{eq:Mel:plan}
\hspace*{-0.5cm}\langle h_1, h_2 \rangle_{x^{2c-1}} = (2\pi)^{-1} \langle \sM_c[h_1], \sM_c[h_2] \rangle_{\IR} \quad \text{ whence } \quad \| h_1\|_{x^{2c-1}}^2=(2\pi)^{-1} \|\sM_c[h]\|_{\IR}^2.
\end{align} 
\paragraph{Usefull Inequality}
The following inequality is due to
\cite{Talagrand1996}, the formulation can be found for
example in \cite{KleinRio2005}.
\begin{lemma}(Talagrand's inequality)\label{tal:re} Let
	$X_1,\dots,X_n$ be independent $\mathcal Z$-valued random variables and let \begin{align*}\bar{\nu}_{h}=n^{-1}\sum_{i=1}^n\left[\nu_{h}(X_i)-\E\left(\nu_{h}(X_i)\right) \right]\end{align*} for $\nu_{h}$ belonging to a countable class $\{\nu_{h},h\in\mathcal H\}$ of measurable functions. Then,
	\begin{align}
	\E(\sup_{h\in\mathcal H}|\overline{\nu}_h|^2-6\Psi^2)_+\leq C \left[\frac{\tau}{n}\exp\left(\frac{-n\Psi^2}{6\tau}\right)+\frac{\psi^2}{n^2}\exp\left(\frac{-K n \Psi}{\psi}\right) \right]\label{tal:re1} 
	\end{align}
	with numerical constants $K=({\sqrt{2}-1})/({21\sqrt{2}})$ and $C>0$ and where
	\begin{equation*}
	\sup_{h\in\mathcal H}\sup_{z\in\mathcal Z}|\nu_{h}(z)|\leq \psi,\qquad \E(\sup_{h\in \mathcal H}|\bar{\nu}_{h}|)\leq \Psi,\qquad \sup_{h\in \mathcal H}\frac{1}{n}\sum_{i=1}^n \Var(\nu_{h}(X_i))\leq \tau.
	\end{equation*}
\end{lemma}
\begin{remark}
	Keeping the bound Eq. \ref{tal:re1}  in mind, let us
	specify particular choices $K$, in fact $K\geq \tfrac{1}{100}$.
	The next bound is now an immediate consequence, 
	\begin{align}
	&\E(\sup_{h\in \mathcal H}|\overline{\nu}_{h}|^2-6\Psi^2)_+\leq C \left(\frac{\tau}{n}\exp\left(\frac{-n\Psi^2}{6\tau}\right)+\frac{\psi^2}{n^2}\exp\left(\frac{-n \Psi}{100\psi}\right) \right)\label{tal:re3} 
	\end{align}
	In the sequel we will make use of the slightly simplified bounds Eq. \ref{tal:re3} rather than Eq. \ref{tal:re1}.
\end{remark}

\section{Proofs of Section \ref{mt}}\label{a:mt}
\begin{proof}[Proof of Proposition \ref{pro:rid:risk:bound}]
First we see that
\begin{align*}
\E_{f_Y}^n(\|f-\widehat f_{k, \xi, r}\|_{x^{2c-1}}^2) &= \|f- \E_{f_Y}^n(\widehat f_{k,\xi, r})\|_{x^{2c-1}}^2 + \E_{f_Y}(\|\widehat f_{k,\xi, r}- \E_{f_Y}^n(\widehat f_{k,\xi, r})\|_{x^{2c-1}}^2) \\
&= \|f- \E_{f_Y}^n(\widehat f_{k,\xi, r})\|_{x^{2c-1}}^2 + \frac{1}{2\pi} \int_{-\infty}^{\infty} \Var_{f_Y}^n(\widehat\sM_c(t))|\mathrm R_{k, \xi, r}(t)|^2dt
\end{align*}
using the Plancherel equality, compare eq \ref{eq:Mel:plan}, and the Fubini-Tonelli theorem.
Considering the bias term, we have for $t\in G_k^c$, $\mathrm R_{k,\xi, r}(t)=\sM_c[g](t)^{-1}$. On the other hand, for $t\in \IR$ we have 
\begin{align*}
|\mathrm R_{k,\xi, r}(t)|=\frac{|\sM_c[g](t)|^{r+1}}{\max(|\sM_c[g](t)|, k^{-1}(1+|t|)^{\xi})^{r+2}} \leq |\sM_c[g](t)|^{-1}.
\end{align*} 
Now the Plancherel equality implies
\begin{align*}
\|f-\E_{f_Y}^n(\widehat f_{k,\xi, r})\|_{x^{2c-1}}^2 &= \frac{1}{2\pi} \int_{-\infty}^{\infty} \left| \sM_c[g](t) \mathrm{R}_{k,\xi, r}(t) - 1\right|^2 |\sM_c[f](t)|^2 dt \\
&=\frac{1}{2\pi} \int_{G_n} \left| \sM_c[g](t) \mathrm{R}_{k,\xi, r}(t) - 1\right|^2 |\sM_c[f](t)|^2 dt\\
&\leq \frac{1}{2\pi} \int_{G_n} |\sM_c[f](t)|^2dt.
\end{align*} 
Now for the variance term, we see directly that
\begin{align*}
\frac{1}{2\pi} \int_{-\infty}^{\infty} \Var_{f_Y}^n(\widehat\sM_c(t)) |\mathrm{R}_{k,\xi, r}(t)|^2 dt &\leq \frac{\sigma_c}{2\pi n} \int_{-\infty}^{\infty} |\mathrm R_{k,\xi, r}(t)|^2 dt.
\end{align*}
which proves the proposition.
\end{proof}

\begin{proof}[Proof of Corollary \ref{cor:rid:consis:ex}]
	To show the corollary, it is sufficient to show that $\|\mathrm R_k\|_{\IR}^2 \leq C_{g,r}k^{2+\gamma^{-1}}$. In detail, we have
	\begin{align*}
	\|\mathrm R_{k} \|_{\IR}^2 = \|\mathrm R_{k} \mathds 1_{G_k}\|_{\IR} + \|\mathrm R_{k} \mathds 1_{G_k^c}\|_{\IR}^2 = \|\mathds 1_{G_k} \sM_c[g]^{r+1} k^{r+2}\|_{\IR}^2 + \|\mathds 1_{G_k^c} \sM_c[g]^{-1}\|_{\IR}^2
	\end{align*}
	using the assumption \textbf{[G1]} and for $r>0\vee (\gamma^{-1}-1)$. The latter restriction ensures that $\sM_c[g]^{r+1} \in \IL^1(\IR)\cap\IL^2(\IR)$. Since $g$ fullfill \textbf{[G1]} we can find positive constants $C_{g,1}, C_{g,2}>0$ only depending on $g$ such that the sets $G_{k,i}:=\IR\setminus [-C_{g,i} k^{\gamma^{-1}}, C_{g,i} k^{\gamma^{-1}}]$ for $i=1, 2$ satisfy the inclusion relationship
	\begin{align*}
	G_{k,1} \subseteq G_k \subseteq G_{k, 2}.
	\end{align*}
	For the first summand we get that 
	\begin{align*}
	k^{2(r+2)}\left\| \mathds 1_{G_{k,2}} \sM_c[g]^{r+1}\right\|_{\IR}^2 = C(g,L,r) k^{2(r+2)} \int_{C_{g,2} k^{\gamma^{-1}}}^{\infty} t^{-2\gamma(r+1)} dt 
	=C_{g, r} k^{2+\gamma^{-1}}
	\end{align*}
	since $\gamma(r+1)>1$. For the second summand we get
	\begin{align*}
	\|\mathds 1_{G_k^c} \sM_c[g]^{-1}\|_{\IR}^2\leq \int_{-C_{g,1}k^{\gamma^{-1}}}^{C_{g,1}k^{\gamma^{-1}}} |\sM_c[g](t)|^{-2}dt \leq C_{g} k^{2+\gamma^{-1}}.
	\end{align*}
	\end{proof}
\begin{proof}[Proof of Theorem \ref{thm:upp:bound}]
	First we see that for $f\in \mathbb D_c^s(L)$ 
	\begin{align*}
	\|\mathds 1_{G_k} \sM_c[f]\|_{\IR}^2 \leq \|\mathds 1_{G_{k,2}} \sM_c[f]\|_{\IR}^2 =\frac{1}{\pi} \int_{C_{g,2} k^{\gamma^{-1}}}^{\infty} |\sM_c[f](t)|^2 dt \leq C(g,L) k^{-2s/\gamma}
	\end{align*}
	staying in the notation of the proof of Corollary \ref{cor:rid:consis:ex}. Further, we have that $\sigma_c=\E_f(X_1^{2(c-1)})\E_g(U_1^{2(c-1)}) \leq C(L,g)$. In total we get
	\begin{align*}
	\E_{f_Y}^n(\|f-\widehat f_{\rho_n, r}\|_{x^{2c-1}}^2) \leq C(g,L, r) (k^{-2s/\gamma} + k^{2+\gamma^{-1}}n^{-1}),
	\end{align*}
	where both summands are balanced by the choice $k_o=n^{\gamma/(2s+2\gamma+1)}$.
	\end{proof}

\section{Proofs of Section \ref{dd}}\label{a:dd}

\begin{proof}[Proof of Theorem\ref{thm:dd:ridge}]
	The proof can be split in two main steps. The first one using a sequence of elementary steps to find a controlable upper bound for the risk of the data-driven estimator. In the second step, we use mainly the Talagrand inequality to show the claim of the theorem. These two steps are expressed through the following lemmata which we state first and proof afterwards.
	\begin{lemma}\label{gde:mme:re:de}
		Under the assumptions of Theorem \ref{thm:dd:ridge} we have for any $k\in \mathcal K_n$,
		\begin{align*}
		\E_{f_Y}^n(\|\widehat f_{\widehat k}-f\|_{x^{2c-1}}^2) \leq &C(\chi_1, \chi_2) \left(\|f-f_{k}\|_{x^{2c-1}}^2+V(k) \right)+108\E_{f_Y}^n(\sup_{k'\in \mathcal K_n} \left(\|\widehat f_{k'} -f_{k'} \|^2_{x^{2c-1}} - \frac{\chi_1}{6}V(k') \right)_+ )\\
		&+C(\chi_1)\E_{f_Y}^n(\sup_{k\in \mathcal K_n} (\widehat V(k)-V(k))_+)
		\end{align*}
		for positive constants $C(\chi_1, \chi_2)$ and $C(\chi_1)$ only depending on $\chi_1$ and $\chi_2$ and $f_k:= \E_{f_Y}^n(\widehat f_k)$.
	\end{lemma}
	 To be able to apply the Talagrand inequality on the term $\E_{f_Y}^n(\sup_{k'\in \mathcal K_n} \left(\|\widehat f_{k'} -f_{k'} \|^2_{x^{2c-1}} - \frac{\chi_1}{6}V(k') \right)_+ )$ we need to split the process first. To do so, let us define the set $\mathbb U:=\{h\in \IL^2(\IRp,  x^{{2 c- 1}}): \|h\|_{x^{{2 c-1}}} \leq 1\}$. Then  for $k\in\mathcal K_n$ we have $\|\widehat f_{k}- f_{k}\|_{ x^{{2c- 1}}} = \sup_{h\in \mathbb U} \langle\widehat f_{k}- f_{k}, h \rangle_{ x^{{2 c-1}}}$ where 
	\begin{align*}
	\langle \widehat f_{k}-f_{k}, h\rangle_{x^{{2c-1}}}=  \frac{1}{2\pi}\int_{-\infty}^{\infty} \left(\widehat{\sM}_c( t)- \E_{f_Y}^n(\widehat{\sM}_{ c}(t))\right) \mathrm R_{k, r}(t) \sM_c[h](-t) d t
	\end{align*}
	by application of the Plancherel equation, eq. \ref{eq:Mel:plan}. Now for a positive sequence $(c_n)_{n\in \mathbb N}$ we decompose the empirical Mellin transform $\widehat{\sM }_{ c}( t), t\in \IR, $ into
	\begin{align*}
	\widehat{\mathcal M}_{ c}( t):&= n^{-1} \sum_{j=1}^n  Y_j^{{ c-1+i t}} \mathds 1_{(0, c_n)}(Y_j^{{ c- 1}})+ n^{-1} \sum_{j=1}^n  Y_j^{{c- 1+it}} \mathds 1_{[c_n, \infty)}(Y_j^{{ c- 1}})\\
	&=: \widehat{\mathcal M}_{c, 1}( t)+\widehat{\mathcal M}_{ c, 2}(t).
	\end{align*}
	Setting 
	\begin{align*}
	\nu_{ k, i}(h):= \frac{1}{2\pi}\int_{-\infty}^{\infty} \left(\widehat{\sM}_{c,i}( t)- \E_{f_Y}^n(\widehat{\sM}_{ c,i}(t))\right) \mathrm R_{k, r}(t) \sM_c[h](-t) d t\quad 
	\end{align*} 
	for  $h\in \mathbb U,\, i\in \{1,2\}$, we can deduce that 
	\begin{align}\label{eq:two:process}
	\hspace*{-0.7cm}\E_{f_Y}^n(\sup_{k\in \mathcal K_n}( \|\widehat f_{k}-f_{k}\|_{x^{2c-1}}^2- \frac{\chi_1}{6}V(k) )_+)  &\leq 2\E_{f_Y}^n(\sup_{k\in \mathcal K_n}( \sup_{h\in \mathbb U} \nu_{k, 1}(h)^2- \frac{\chi_1}{12} V(k) )_+) 
	+ 2 \E_{f_Y}^n( \sup_{k\in \mathcal K_n} \sup_{h\in \mathbb U}\nu_{k,2}(h)^2)).
	\end{align}
	This decompostion and the following Lemma then proves the claim.
	\begin{lemma}\label{gde:mme:re:ta}
		Under the assumptions of Theorem \ref{thm:dd:ridge} 
		\begin{align*}
		(i) \quad &\E_{f_Y}^n(\sup_{k\in \mathcal K_n}(\sup_{h\in \mathbb U} \nu_{k, 1}(h)^2- \frac{\chi_1}{12} V(k) )_+) \leq \frac{C({g, r, \E_f(X_1^{2(c-1)})}}{n}, \\
		(ii)\quad & \E_{f_Y}^n( \sup_{k\in \mathcal K_n} \sup_{h\in \mathbb U}\nu_{k,2}(h)^2)) \leq \frac{C(\sigma_c,\E_{f_Y}(Y_1^{5(c-1)}))}{n} \text{ and }\\
		(iii) \quad & \E_{f_Y}^n(\sup_{k\in \mathcal K_n} (\widehat V(k)-V(k))_+)\leq \frac{C(\sigma_c,\E_{f_Y}(Y_1^{4(c-1)}))}{n}. 
		\end{align*}
	\end{lemma}
\end{proof}

\begin{proof}[Proof of Lemma \ref{gde:mme:re:de}]
		Since $\chi_2\geq \chi_1$ and  by the definition of $\widehat{k}$ follows for any $k\in \mathcal K_n$,
	\begin{align*}
	\|f-\widehat f_{\widehat{k}}\|_{x^{2c-1}}^2 &\leq 3\|f-\widehat f_{k}\|_{ x^{2c-1}}^2 +3 \|\widehat f_{k}-\widehat f_{k \wedge \widehat{k}}\|_{x^{2c-1}}^2+ 3\|\widehat f_{k \wedge \widehat{k}}-\widehat f_{\widehat{k}}\|_{x^{2c-1}}^2 \\
	&\leq 3\|f-\widehat f_{k}\|_{x^{2c-1}}^2 + 3( \widehat A(\widehat {k})+ \chi_1\widehat V(k)+ \widehat A(k)+ \chi_1\widehat V(\widehat{k})) \\
	& \leq 3\|f-\widehat f_{k}\|_{x^{2c-1}}^2 + 3(2 \widehat A(k)+ (\chi_1+\chi_2)\widehat V(k)).
	\end{align*}
	To simplify the notation, let us set $\chi:= (\chi_1+\chi_2)/2$.
	Let us now have a closer look at $\widehat A(k)$. From \begin{align*}
	\|\widehat f_{k'} - \widehat f_{k'\wedge k} \|_{x^{2c-1}}^2  &\leq 3 (\|\widehat f_{k'} -f_{k'} \|^2_{x^{2c-1}} +\|\widehat f_{k'\wedge k} -f_{k'\wedge k} \|^2_{x^{2c-1}} +\| f_{k‘} -f_{k'\wedge k} \|^2_{x^{2c-1}} ) \\
	& \leq 6 \|\widehat f_{k'} -f_{k'} \|^2_{x^{2c-1}} + 3\| f -f_{k} \|^2_{x^{2c-1}} 
	\end{align*} we conclude by a straight forward calculation that
	\begin{align*}
	\widehat A(k) \leq &6\sup_{k'\in \mathcal K_n} \left(\|\widehat f_{k'} -f_{k'} \|^2_{x^{2c-1}} - \frac{\chi_1}{6}V(k') \right)_+
	+3\| f - f_{k}\|_{x^{2c-1}}^2 + \chi_1\sup_{k'\in \mathcal K_n} ( V(k')-\widehat V(k'))_+.
	\end{align*}
	This implies 
	\begin{align*}
	\hspace*{-0.8cm}\E_{f_Y}^n(\|f-\widehat f_{\widehat{k}}\|_{x^{2c-1}}^2 ) \leq C(\chi) &\left(\|f-f_{k}\|_{x^{2c-1}}^2+V(k) \right)+108\E_{f_Y}^n(\sup_{k'\in \mathcal K_n} (\|\widehat f_{k'} -f_{k'} \|^2_{x^{2c-1}} - \frac{\chi_1}{6}V(k') )_+ )\\
	&+C(\chi_1)\E_{f_Y}^n(\sup_{k\in \mathcal K_n} (V(k)-\widehat V(k))_+) .
	\end{align*}
	\end{proof}

\begin{proof}[Proof of Lemma \ref{gde:mme:re:ta}]
	\underline{Proof of $(i)$:} 
	Now on the first summand of the right hand side of \ref{eq:two:process} we can apply the Talagrand. Let us start with the first term. We use that
	\begin{align*}
	\E_{f_Y}^n(\sup_{k\in \mathcal K_n}( \sup_{h\in \mathbb U} \nu_{k, 1}(h)^2- \frac{\chi_1}{12} V(k) )_+)  \leq \sum_{k=1}^{K_n} \E_{f_Y}^n(( \sup_{h\in \mathbb U} \nu_{k, 1}(h)^2- \frac{\chi_1}{12} V(k) )_+) 
	\end{align*}
	where $K_n:=\max(\mathcal K_n)$.
	To apply now the Talagrand inequality, compare Lemma \ref{tal:re},  to each summand we need to determine the constants $\Psi^2, \psi^2$ and $\tau$ first. Staying in the notation of the Talagrand inequality, we set for  $h\in \mathbb U$, 
	\begin{align*}
	\nu_h(y):= \frac{1}{2\pi} \int_{-\infty}^{\infty} y^{c-1+it} \mathds 1_{(0, c_n)}(y) \mathrm R_{k}(t) \sM_c[h](-t)dt,\quad y\in \IRp.
	\end{align*}
	 Now applying Cauchy-Schwartz inequality 
	\begin{align*}
\nu_{k,1}^2(h) \leq \frac{\|h\|_{x^{2c-1}}^2}{2\pi} \int_{-\infty}^{\infty} |\widehat\sM_{c,1}(t)- \E_{f_Y}^n(\widehat\sM_{c,1}(t))|^2 |\mathrm{R}_k(t)|^2dt  \leq \frac{1}{2\pi} \int_{-\infty}^{\infty} |\widehat\sM_{c,1}(t)- \E_{f_Y}^n(\widehat\sM_{c,1}(t))|^2 |\mathrm{R}_k(t)|^2dt  
	\end{align*} since $h\in \mathbb U$. We deduce that
	\begin{align*}
	\E_{f_Y}^n(\sup_{h\in \mathbb U} \nu_{k,1}(h)^2) \leq \frac{1}{2\pi} \int_{-\infty}^{\infty} \E_{f_Y}^n(|\widehat\sM_{c,1}(t)- \E_{f_Y}^n(\widehat\sM_{c,1}(t))|^2) |\mathrm{R}_k(t)|^2dt  \leq \sigma_c \|\mathrm R_k\|_{\IR}^2n^{-1}=:\Psi^2,
	\end{align*}
	compare proof of Proposition \ref{pro:rid:risk:bound}. For $y>0$ we have $|\nu_h(y)|^2\leq c_n^2 \|\mathrm R_{k}\|_{\IR}^2\|\sM_c[h]\|_{\IR}^2/(2\pi) \leq c_n^2\|\mathrm R_{k}\|_{\IR}^2=:\psi^2$ since $h\in \mathbb U$. Additionally, we have for any $h\in \mathbb U$ that
	$\Var_{f_Y}^n(\nu_h(Y_1))\leq \E_{f_Y}^n(\nu_h^2(Y_1)) \leq \|f_Y\|_{\infty, x^{2c-1}} \|\nu_h\|_{x^{1-2c}}^2$. More precisely, we see that 
	\begin{align*}
	y^{2c-1} \int_0^{\infty} f(x) g(y/x)x^{-1}dx \leq \| g\|_{\infty, x^{2c-1}} \E(X_1^{2(c-1)}), \quad y \in \IRp.
	\end{align*}
	Next, we have 
	\begin{align*}
		\|\nu_h\|_{x^{1-2c}}^2 \leq \frac{1}{2\pi} \int_{-\infty}^{\infty} |\sM_c[h](t)|^2 |\mathrm R_{k}(t)|^2 dt \leq \|\mathrm R_{k}^2\|_{\infty} \frac{1}{2\pi}\|\sM_c[h]\|_{\IR}^2 \leq  \|\mathrm R_{k}^2\|_{\infty}
		\end{align*}
		which implies the choice $\tau:= \|g\|_{\infty, x^{2c-1} } \E_f(X_1^{2(c-1)}) \|\mathrm R_{k}^2\|_{\infty}$.	Applying now the Talagrand inequality we get\begin{align*}
		\E_{f_Y}^n((\sup_{h\in \Uz} \overline\nu_h^2 -6\Psi^2)_+) &\leq \frac{C_{f_Y}}{n} \left(\|\mathrm R_{k}^2\|_{\infty} \exp(-C_{f_Y} \frac{\|\mathrm R_k\|_{\IR}^2}{\|\mathrm R_{k}^2\|_{\infty}}) + c_n^2 \exp(-\frac{\sqrt{n\sigma_c}}{100 c_n})\right) \\
		& \leq \frac{C_{f_Y}}{n} \left( \|\mathrm R_{k}^2\|_{\infty} \exp(-C_{f_Y} \frac{\|\mathrm R_k\|_{\IR}^2}{\|\mathrm R_{k}^2 \|_{\infty}}) + n^{-1} \right)
		\end{align*}
		for the choice $c_n:= \sqrt{n\sigma}/(100\log(n^{2}))$. Following the same step as in the proof of \ref{cor:rid:consis:ex}, we can state that $K_n\leq C_{g,r} n^{\gamma/(2\gamma+1)}\leq C_{g,r} n^{1}$. For $\chi_1 \geq 72$ we can conclude that
		\begin{align*}
		\hspace*{-0.7cm}\E_{f_Y}^n(\sup_{k \in \mathcal K_n} \left( \sup_{h\in \Uz} \nu_{k, 1}(h)^2- \frac{\chi_1}{12} V(k) \right)_+)  &\leq \sum_{k=1}^{K_n}\frac{C_{f_Y}}{n} \left( \|\mathrm R_{k}^2\|_{\infty} \exp(-C_{f_Y} \frac{\|\mathrm R_k\|_{\IR}^2}{\|\mathrm R_{k}^2\|_{\infty}}) + n^{-1} \right)\\
		&\leq \frac{C_{f_Y}}{n} (1 + \sum_{k=1}^{K_n} \|\mathrm R_{k}^2\|_{\infty} \exp(-C_{f_Y} \frac{\|\mathrm R_k\|_{\IR}^2}{\|\mathrm R_{k}^2\|_{\infty}})). 
		\end{align*}
		Now it can easily be seen that there exists constants $c_{g,r}, C_{g,r}>0$ such that $ c_{g,r}k^{2+\gamma^{-1}} \leq \|\mathrm R_k\|_{\IR}^2\leq C_g k^{2+\gamma^{-1}}$ using \textbf{[G1]}. By  simple calculus, one can show that $\|\mathrm R_{k}^2\|_{\infty} \leq C_{g,r} k^{2}$. Since $(k^2 \exp(-C_{f_Y} k^{\gamma^{-1}}))_{k\in \mathbb N}$ is summable we can deduce that $\E_{f_Y}^n(\sup_{k \in \mathcal K_n} \left( \sup_{h\in \Uz} \nu_{k, 1}(h)^2- \frac{\chi_1}{12} V(k) \right)_+)  \leq  C_{f_Y} n^{-1}.$\\
		\underline{Let us now show part $(ii):$} For any $h\in \mathbb U$ and $k\in \mathcal K_n$ we get \begin{align*}
		\nu_{k,2}(h)^2 \leq \frac{\|h\|_{x^{2c-1}}^2}{2\pi} \int_{-\infty}^{\infty}|\widehat\sM_{c, 2}(t)- \E_{f_Y}^n(\widehat\sM_{c, 2}(t)|^2 |\mathrm R_{k}(t)|^2 dt  \leq  \frac{1}{2\pi} \int_{-\infty}^{\infty}|\widehat\sM_{c, 2}(t)- \E_{f_Y}^n(\widehat\sM_{c, 2}(t))|^2 |\mathrm R_{K_n}(t)|^2 dt 
		\end{align*}
		and thus, since $\mathrm R_{k}(t) \geq \mathrm R_{\ell}(t)$ for all $t\in \IR$ and $k\geq \ell$, 
		\begin{align*}
		\E_{f_Y}^n( \sup_{k\in \mathcal K_n} \sup_{h\in \mathbb U}\nu_{k,2}(h)^2))  \leq \frac{1}{n}\|\mathrm R_{K_n}\|_{\IR}^2 \E_{f_Y}(Y_1^{2(c-1)} \mathds 1_{[c_n, \infty)}(Y_1^{c-1})).
		\end{align*}
		Now by definition of $\mathcal K_n$ we know that $\|\mathrm R_{K_n}\|_{\IR}^2 n^{-1} \leq 1$. We deduce that for any $p>0$ 
		\begin{align*}
		\E_{f_Y}^n ( \sup_{m\in \mathcal K_n} \sup_{h\in \mathbb U}\nu_{k,2}(h)^2)) \leq c_n^{-p} \E_{f_Y}(Y_1^{(2+p)(c-1)}) \leq \frac{C(\sigma_c,\E_{f_Y}(Y_1^{5(c-1)}))}{n}
		\end{align*}
		choosing $p=3$ and by the definition of $(c_n)_{n\in \mathbb N}$.\\
		\underline{Part $(iii)$:} First we see that for any $k\in \mathcal K_n$,  $(V(k)-\widehat V(k))_+ =\|\mathrm R_k\|_{\IR}^2n^{-1}(\sigma_c-2\widehat\sigma_c)_+\leq (\sigma_c-2\widehat\sigma_c)_+$.
		On $\Omega=\{|\widehat\sigma_c-\sigma_c|\leq \sigma_c/2\}$ we have  $\frac{\sigma_c}{2} \leq \widehat{\sigma}_c\leq \frac{3}{2} \sigma_c$. This implies
		\begin{align*}
		\E_{f_Y}^n(\sup_{k\in \mathcal K_n} (V(k)-\widehat V(k))_+)\leq \E_{f_Y}^n((\sigma_c-2\widehat\sigma_c)_+)\leq 2\E_{f_Y}^n(|\sigma_c-\widehat\sigma_c|\mathds 1_{\Omega^c})\leq 4 \frac{\Var_{f_Y}^n(\widehat\sigma_c)}{\sigma_c}
		\end{align*}
		applying the Cauchy Schwartz inequality and the Markov inequality. Now the last inequality implies the claim.
	\end{proof}

\section*{References}

\bibliographystyle{myjmva}
\bibliography{MuDERA}

\end{document}